\documentclass[12pt]{amsart}
\usepackage[T1]{fontenc}
\usepackage[latin9]{inputenc}
\usepackage{geometry}
\usepackage{mathrsfs}
\usepackage{amstext}
\usepackage{amsthm}
\usepackage{amssymb}
\usepackage{bbm}

\usepackage[svgnames]{xcolor}
\usepackage[bookmarksnumbered=true]{hyperref} 
\hypersetup{
	colorlinks = true,
	linkcolor = Blue,
	anchorcolor = blue,
	citecolor = Green,
	filecolor = blue,
	urlcolor = FireBrick
}
\usepackage[nameinlink]{cleveref}

\makeatletter
\numberwithin{equation}{section}
\numberwithin{figure}{section}

\theoremstyle{plain}
\newtheorem{theorem}{Theorem}[section]
\theoremstyle{definition}
\newtheorem{definition}[theorem]{Definition}
\theoremstyle{definition}

\newtheorem{remark}[theorem]{Remark}
\theoremstyle{plain}

\theoremstyle{plain}
\newtheorem{lemma}[theorem]{Lemma}
\theoremstyle{plain}
\newtheorem{proposition}[theorem]{Proposition}
\theoremstyle{definition}
\newtheorem{example}[theorem]{Example}
\theoremstyle{definition}

\theoremstyle{definition}

\newtheorem*{definition*}{Definition}
\newtheorem*{proposition*}{Proposition} 
\newtheorem*{remark*}{Remark}
\newtheorem*{example*}{Example}
\newtheorem*{problem*}{Problem}
\newtheorem*{question*}{Question}
\newtheorem*{theorem*}{Theorem}
\newtheorem*{lemma*}{Lemma}
\newtheorem*{corollary*}{Corollary}
\newtheorem*{Acknowledgments*}{Acknowledgments}

\renewenvironment{proof}[1][\proofname]{\medskip \noindent {\bfseries #1. }}{\hfill \qedsymbol\medskip}

\MakeRobust{\ref}
\newcommand{\labeltext}[2]{
\@bsphack
\csname phantomsection\endcsname
\def\@currentlabel{#1}{\label{#2}}
\@esphack
}

\usepackage{scalerel}
\usepackage[usestackEOL]{stackengine}
\def\dashint{\,\ThisStyle{\ensurestackMath{%
  \stackinset{c}{.2\LMpt}{c}{.5\LMpt}{\SavedStyle-}{\SavedStyle\phantom{\int}}}%
  \setbox0=\hbox{$\SavedStyle\int\,$}\kern-\wd0}\int}

\usepackage{enumitem}

\DeclareRobustCommand{\SkipTocEntry}[5]{}

\newcommand{\mQ}{\mathbb{Q}}   
\newcommand{\mR}{\mathbb{R}}   
\newcommand{\mC}{\mathbb{C}}   
\newcommand{\mN}{\mathbb{N}}   

\newcommand{\abs}[1]{\lvert #1 \rvert}  
\newcommand{\norm}[1]{\lVert #1 \rVert}  

\newcommand{\ol}[1]{\overline{#1}}

\newcommand{\mD}{\mathscr{D}}

\newcommand{\mH}{\mathcal{H}}

\newcommand{\mL}{\mathcal{L}}

\newcommand{\calS}{\mathcal{S}}

\newcommand{\bfi}{\mathbf{i}}

\newcommand{\id}{\mathrm{Id}}

\newcommand{\p}{\partial}

\DeclareMathOperator{\supp}{supp}

\newcommand{\rmd}{\mathrm{d}}

\makeatother

\usepackage{babel}

\usepackage{orcidlink}

\begin{document}

\title[On scattering behavior of corner domains: part II]{On scattering behavior of corner domains with anisotropic inhomogeneities: part II}  

\begin{sloppypar}

\begin{abstract}
We study the scattering behavior of an anisotropic inhomogeneous Lipschitz medium at a fixed wave number, continuing our previous work \cite{KSS23Anisotropic} and using free boundary techniques from \cite{SS25vanishingcontrast}. Our main results can be categorized into two distinct cases. 
In the first case, we show that in two dimensions, piecewise $C^{1}$ or convex penetrable obstacles with corners, and in higher dimensions, obstacles with edge points, always induce nontrivial scattering for any incoming wave. In the second case, we prove that piecewise $C^{1}$ obstacles with corners in two dimensions (and with edge points in higher dimensions) with angles $\notin\pi\mQ$ always produce nontrivial scattering for any incoming wave. 
\end{abstract}

\subjclass[2020]{35J15, 35P25, 35R35}
\keywords{free boundary, corners always scatter, harmonic polynomials, blowup solutions}

\author[Kow]{Pu-Zhao Kow\,\orcidlink{0000-0002-2990-3591}}
\address{Department of Mathematical Sciences, National Chengchi University, Taipei 116, Taiwan}
\email{pzkow@g.nccu.edu.tw} 

\author[Salo]{Mikko Salo\,\orcidlink{0000-0002-3681-6779}}
\address{Department of Mathematics and Statistics, P.O. Box 35 (MaD), FI-40014 University of Jyv\"{a}skyl\"{a}, Finland}
\email{mikko.j.salo@jyu.fi}

\author[Shahgholian]{Henrik Shahgholian\,\orcidlink{0000-0002-1316-7913}}
\address{Department of Mathematics, KTH Royal Institute of Technology, SE-10044 Stockholm, Sweden}
\email{henriksh@kth.se}

\maketitle

\tableofcontents{}

\section{Introduction}

\subsection{Mathematical model}

Let $D$ be a bounded Lipschitz domain in $\mR^{n}$ (where $n\ge 2$) such that $\mR^{n}\setminus\overline{D}$ is connected. Within this domain, let $\rho\in L^{\infty}(D)$ be a positive real-valued function. Additionally, let $A\in (C^{0,1}(\overline{D}))^{n\times n}$ be a real  matrix-valued symmetric function, satisfying the condition of uniform ellipticity 
\begin{equation}
c_{\rm ellip}^{-1}\abs{\xi}^{2} \le \xi\cdot A(x)\xi \le c_{\rm ellip}\abs{\xi}^{2} \quad \text{for a.e. $x\in D$ and all $\xi\in\mR^{n}$} \label{eq:ellipticity}
\end{equation}
for some constant $c_{\rm ellip}>0$. Under the assumption that the medium outside $D$ is homogeneous, if we illuminate the anisotropic medium $(D,A,\rho)$ with an incident field $u^{\rm inc}$ having a fixed wave number $\kappa>0$ that satisfies 
\begin{equation}
(\Delta+\kappa^{2})u^{\rm inc}=0 \quad \text{in $\mR^{n}$,} \label{eq:incident}
\end{equation}
classical scattering theory (see, e.g. \cite{CCH23InverseScatteringTransmission,CK19scattering,KG08Factorization}) guarantees the existence of unique scattered field $u^{\rm sc} \in H_{\rm loc}^{1}(\mR^{n})$ that is outgoing (i.e., satisfies the outgoing Sommerfeld radiation condition \cite[Definition~1.1]{KSS23Anisotropic}). The total field $u^{\rm to}=u^{\rm sc}+u^{\rm inc}$ satisfies 
\begin{equation*}
\left(\nabla\cdot\tilde{A}(x)\nabla + \kappa^{2}\tilde{\rho}(x)\right) u^{\rm to} = 0 \quad \text{in $\mR^{n}$,} 
\end{equation*}
where 
\begin{equation*}
\tilde{A} = A\chi_{\overline D} + \id\chi_{\mR^{n}\setminus \overline{D}} \quad \text{and} \quad \tilde{\rho} = \rho\chi_{\overline{D}} + \chi_{\mR^{n}\setminus \overline{D}}. 
\end{equation*}

\begin{remark}
For notational convenience, throughout the paper we use $A$ to denote $\tilde A$.
\end{remark}

If $A(x_0) \neq \id$ or $\rho(x_0) \neq 1$ at some point $x_0 \in \p D$, then one might expect that the obstacle $(D, A, \rho)$ is strong enough to produce scattering effects for many incident waves. Moreover, if the boundary of $D$ has singularity at $x_0$, it might be true that every incident wave scatters nontrivially (i.e.\ for any $u^{\mathrm{inc}}$ solving \eqref{eq:incident},  $u^{\mathrm{sc}}|_{\mR^n \setminus \ol{D}}$ is not identically zero).

Conversely, if the anisotropic medium $(D,A,\rho)$ is nonscattering with respect to an incident field $u^{\rm inc}$ in the sense that $u^{\rm sc}=0$ in $\mR^{n}\setminus\overline{D}$, then the scattered field $u^{\rm sc}$ satisfies  
\begin{equation}
\left\{\begin{aligned}
& (\nabla\cdot A(x)\nabla+\kappa^{2}\rho(x))u^{\rm sc} = -(\nabla\cdot (A(x)-\id)\nabla+\kappa^{2}(\rho(x)-1))u^{\rm inc} && \text{in $D$,} \\ 
& u^{\rm sc}=0 ,\quad \nu\cdot A\nabla u^{\rm sc} = \nu\cdot(\id-A)\nabla u^{\rm inc} && \text{on $\partial D$,}  
\end{aligned}\right. \label{eq:scattered-field}
\end{equation}
where $\nu$ is the inward pointing  unit normal vector to $\partial D$. One could then hope to show that $D$ cannot have boundary singularities. In the earlier works \cite{CVX23RegularityITEP, KSS23Anisotropic},  such results were established for some anisotropic scatterers by using the free boundary techniques developed in \cite{CakoniVogelius,SS21NonscatteringFreeBoundary} for nonscattering problems (see also \cite{AI96Conductivity} for the Calder{\'o}n problem). In the present article we continue this line of research by covering cases not included in \cite{KSS23Anisotropic}, based on the new methods introduced in \cite{SS25vanishingcontrast}.

\subsection{Definitions and terminology} 

Before we proceed, let us recall the key definitions, assumptions, and formalism pertaining to corner domains, as referenced in the paper's title. 
Suppose that $x_{0}\in\partial D$ is the boundary point of interest. In some of our results, we assume that the contrast 
\begin{equation}
h(x) := \kappa^{2}(\rho(x)-1)\chi_{D} \label{eq:contrast-h}
\end{equation}
satisfies the following non-degeneracy condition at $x_{0}\in\partial D$: 
\begin{equation}
\text{there is $r>0$ such that $h\in C^{\alpha}(\overline{D}\cap B_{r}(x_{0}))$ and $h(x_{0})\neq 0$.} \label{eq:non-degeneracy}
\end{equation}
Recall the following definition for $n=2$.

\begin{definition}
Let $k\in\mN\cup\{\infty\}$. An open set $D$ in $\mR^{2}$ is said to have a \emph{piecewise $C^{k}$ boundary} if, for any $x_{0}$, there exist $\delta>0$ and a rigid motion $\Phi:B_{\delta}(x_{0})\rightarrow V$ into an open set $V\subset\mR^{2}$ such that $\Phi(x_{0})=0$ and 
\begin{equation*}
\Phi(D\cap B_{\delta}(x_{0})) = \{x_{2}>\eta(x_{1})\}\cap B_{r}(0), 
\end{equation*}
where $\eta\in C([-r,r])$ and both $\eta|_{[-r,0]} $ and $\eta|_{[0,r]}$ are $C^{k}$-functions up to the endpoints of the respective closed intervals. 
\end{definition}

We also recall the following definition for $n\ge 3$. 

\begin{definition} 
Let $D$ be an open set in $\mR^{n}$ with $n\ge 3$. A boundary point $x_{0}\in\partial D$ is said to be an \emph{edge point} if there is $\delta>0$ and a $C^{1}$ diffeomorphism $\Phi: B_{\delta}(x_{0})\rightarrow V$ onto an open set $V\subset\mR^{n}$ such that $\Phi(x_{0})=0$, $(\nabla\otimes\Phi)(x_{0})$ is a rotation matrix, and 
\begin{equation*}
\Phi(\overline{D}\cap B_{\delta}(x_{0})) = (S\times\mR^{n-2}) \cap \Phi(B_{\delta}(x_{0})), 
\end{equation*}
where $(\nabla\otimes\Phi)_{ij}=\partial_{i}\Phi_{j}$, $S=\{(r\cos\theta,r\sin\theta):r>0,0\le\theta\le\theta_{0}\}$ is a closed sector with angle $\theta_{0}\in(0,2\pi)\setminus\{\pi\}$. In this case, we also refer to the edge point as having angle $\theta_{0}$. 
\end{definition}

\subsection{Methodology} \label{subsec_methodology}

As noted in \cite{KSS23Anisotropic}, the anisotropic nonscattering problem \eqref{eq:scattered-field} can be reformulated as a Bernoulli-type free boundary problem of the form 
\begin{subequations} \label{eq:scattered-field-distribution-form-recall}
\begin{equation}
(\nabla\cdot A(x)\nabla+\kappa^{2}\rho(x))u^{\rm sc} = f \mL^{n}\lfloor D + g \mH^{n-1}\lfloor\partial D \text{ in $\mR^n$},\quad u^{\rm sc}|_{\mR^{n}\setminus\overline{D}}=0, 
\end{equation}
where 
\begin{equation}
f = -(\nabla\cdot (A(x)-\id)\nabla+h)u^{\rm inc} ,\quad g = \nu\cdot(A-\id)\nabla u^{\rm inc} \label{eq:coefficient-f}
\end{equation}
\end{subequations} 
and $h(x) := \kappa^{2}(\rho(x)-1)\chi_{D}$, see \eqref{eq:scattered-field-distribution-form} below. 
Hereafter, we denote $\mL^{n}\lfloor D \equiv \chi_{D}$ the $n$-dimensional Lebesgue measure restricted to $D$, and by $\mH^{n-1}\lfloor\Gamma$ the $(n-1)$-dimensional Hausdorff measure restricted to $\Gamma$. 
Our main strategy is to consider the blowup limit of problem \eqref{eq:scattered-field-distribution-form-recall}. To do this, we suppose that near a boundary point $x_{0}=0$ (say) the quantities $u^{\rm inc}$, $A-\id$, and $\rho-1$ have in $\overline{D}$ the expansions 
\begin{align*}
u^{\rm inc} &= H + R^{(m)}, \\
A-\id &= B + R^{(b)}, \\
\rho-1 &= \eta + R^{(p)},
\end{align*}
where $H, B, \eta$ are homogeneous polynomials of degree $m$, $b$ and $p$, respectively, and $H$ is harmonic (since $u^{\rm inc}$ solves $(\Delta+\kappa^2) u^{\rm inc} = 0$). The remainder terms are assumed to vanish at $0$ faster than the main terms, together with relevant derivative bounds.


In the discussion below we will assume that $m \geq 2$ to exclude certain special cases. It follows that near $0$, one has in $\overline D$ 
\begin{align*}
f &= -\nabla \cdot B \nabla H - \kappa^2 \eta H + R^{(f)}, \\
g &= \nu \cdot B \nabla H + R^{(g)},
\end{align*} 
where $R^{(f)}$ and $R^{(g)}$ vanish at $0$ faster than the other terms in $f$ and  $g$, respectively. Note that $\nabla \cdot B \nabla H$ has degree $m+b-2$, $\eta H$ has degree $m+p$, and $B \nabla H$ has degree $m+b-1$. 
We classify the blowup equation corresponding to \eqref{eq:scattered-field-distribution-form-recall} into two cases: $b\ge p+2$ and $0 \le b<p+2$. Theorems for the case $b\ge p+2$ are stated in \Cref{sec:C11-blowup}, followed by the case $0 \le b<p+2$ in \Cref{sec:blowup-Bernoulli}.

When $b\ge p+2$, we define the blowup sequence $u_r(x) = u^{\rm sc}(r x)/r^{m+p+2}$ at 0 (when it exists). The corresponding blowup equation is 
\[
\Delta u = -\kappa^2 \eta H \chi_{\{ u \neq 0 \}}.
\]
In this scenario, only the contribution from $\rho-1$ survives in the blowup, while the Bernoulli condition vanishes. 

In the case $0 \le b<p+2$, we define the blowup sequence $u_r(x) = u^{\rm sc}(rx)/r^{m+b}$ at 0 (when it exists) that  leads to the blowup equation 
\[
\Delta u = - (\nabla \cdot B \nabla H) \mL^{n}\lfloor D_{0} + (B \nabla H \cdot \nu_0)\mH^{n-1}\lfloor \partial D_{0}, \qquad u|_{\mR^n \setminus \ol{D}_0} = 0,
\]
where $D_0$ is the blowup of $D$ at $x_{0}=0$ (assuming it exists), and $\nu_0$ is the unit outer normal to $\partial D_0$. In this case the contribution from $\rho-1$ disappears in the blowup, but the Bernoulli condition survives.

\subsection{Statements of the main results and related work\label{sec:main-theorems}} 

The discussion in Section \ref{subsec_methodology} indicates that the results will be divided into several cases according to the behavior of $A$, $\rho$ and $u^{\rm inc}$ at $x_{0}$. Before presenting the formal statements of the main results, we briefly highlight them as follows. We say that $A(x_0) \cong \id$ if $\abs{A(x)-\id}\le C\abs{x-x_{0}}^{2+\alpha}$ and $\abs{\nabla A(x)}\le C\abs{x-x_{0}}^{1+\alpha}$, 
near $x_0$ in $\overline D$, for some $\alpha > 0$.

\begin{itemize}
\item If $A(x_{0}) \cong \id$ and $\rho(x_{0})\neq 1$, any corner with opening angle $\neq\pi$ scatters every incident wave nontrivially (\Cref{thm:1}). 
\item If $A(x_{0}) \cong \id$ but $\rho(x_{0}) = 1$, corners of any angle may be invisible to certain incident waves (\Cref{exa:vanishing-condition}).  
\item If $A(x_{0})\neq\id$ and $\rho(x_{0})=1$, corners with angles $\notin\mQ\pi$ scatter every incident wave nontrivially (\Cref{thm:3,thm:4}), whereas corners with angles $\in\mQ\pi$ may be invisible for some incident waves (\cite[Example~1.5]{KSS23Anisotropic}). 
\item The special case $A=a\id$ and $\rho=a$ with $a\neq 1$, which satisfies $A(x_{0})\neq\id$ and $\rho(x_{0})\neq 1$, was studied in \cite{CHLX24examples}, showing that the scattering behavior is linked to the interior transmission eigenvalue problem. 
\end{itemize}
We now proceed to the formal statements of the main results.

\subsubsection{Vanishing Bernoulli condition\label{sec:C11-blowup}}

Our first theorem states that convex\footnote{We recall that bounded convex domains always have Lipschitz boundary \cite[Corollary~1.2.2.3]{Grisvard2011PDE}.} or any  piecewise $C^{1}$ planar obstacle whose boundary has a singular (non-$C^{1}$) point always scatter: 

\begin{theorem}\label{thm:1}
Let $D$ be a bounded open set in $\mR^{2}$, and suppose that 
\begin{itemize}
\item either $D$ is convex, or 
\item $D$ is simply connected and has piecewise $C^{1}$ boundary. 
\end{itemize}
Let $\kappa>0$, let $\rho\in L^{\infty}(D)$ be a positive real-valued function, and let $A\in (C^{1}
(\overline{D}))^{n\times n}$ be a real symmetric matrix-valued function, satisfying the condition of uniform ellipticity \eqref{eq:ellipticity}. Suppose that there exists $x_{0}\in\partial D$ such that 
\begin{enumerate}
\renewcommand{\labelenumi}{\theenumi}
\renewcommand{\theenumi}{(\roman{enumi})}
\item \label{itm:nondegeneracy-contrast} the contrast $h$ given in \eqref{eq:contrast-h} satisfies the non-degeneracy condition \eqref{eq:non-degeneracy} at $x_{0}$, and 
\item \label{itm:strong-decay1} the coefficient $A$ satisfies $\abs{A(x)-\id}\le C\abs{x-x_{0}}^{2+\alpha}$ and $\abs{\nabla A(x)}\le C\abs{x-x_{0}}^{1+\alpha}$.
\end{enumerate}
If $\partial D$ is not $C^{1}$ near $x_{0}$, then the anisotropic medium $(D,A,\rho)$ scatters every incident wave nontrivially. 
\end{theorem}

It is worth remarking that it seems plausible that  the assumptions $\abs{A(x)-\id}\le C\abs{x-x_{0}}^{2+\alpha}$ and $\abs{\nabla A(x)}\le C\abs{x-x_{0}}^{1+\alpha}$, can be relaxed to 
\begin{equation}\label{Dini}
\abs{A(x)-\id}=  \omega(|x - x_0|)\abs{x-x_{0}}^{2}, \qquad \abs{\nabla A(x)}=  \omega(|x - x_0|)\abs{x-x_{0}}^1
\end{equation}
where $\omega (r) $ is double-Dini, which is used for the monotonicity formula in this paper. Since this would be  a more technical result, we leave it out in this paper.

\begin{example}[Necessity of condition \eqref{eq:non-degeneracy}] \label{exa:vanishing-condition}

In the trivial example where $A \equiv \id$ and $\rho \equiv 1$, the obstacle $D$ is not present and no incident wave produces scattering effects (i.e.\ $u^{\rm sc} \equiv 0$ for any incident wave). In particular, $D$ can have corners of any type. This shows that it is not possible to drop the assumption \eqref{eq:non-degeneracy} in the theorem above.

Another related example is given in \cite[Section~3]{CVX23RegularityITEP} or \cite[Example~1.6]{KSS23Anisotropic} based on diffeomorphism invariance. Let $D$ be a bounded domain in $\mR^{2}$, with piecewise $C^{\infty}$ boundary, with a corner $x_{0}\in\partial D$. Let $\Phi:D\rightarrow D$ be a diffeomorphism such that $\Phi\in C^{\infty}(\overline{D})$, $\Phi^{-1}\in C^{\infty}(\overline{D})$, $\Phi(x)=x$ for all $x\in\partial D$ and $(\nabla\otimes\Phi)(x_{0})=\id$. Let 
\begin{equation*}
A=\Phi_{*}(\id) = \frac{(\nabla\otimes\Phi)(\nabla\otimes\Phi)^{\intercal}}{\abs{\det(\nabla\otimes\Phi)}}\circ\Phi^{-1} \in (C^{\infty}
(\overline{D}))^{2\times 2} 
\end{equation*} 
and 
\begin{equation*}
\rho=\Phi_{*}(1) = \frac{1}{\abs{\det(\nabla\otimes\Phi)}}\circ\Phi^{-1} \in C^{\infty}(\overline{D}) 
\end{equation*}
be the pushforwards by $\Phi$, which satisfy $A(x_{0})=\id$ and $\rho(x_{0})=1$. Let $u^{\rm inc}\not\equiv 0$ solves \eqref{eq:incident}. Choosing $v=u^{\rm inc}|_{D}$ and $u^{\rm to}=\Phi_{*}v$ gives a pair $(u^{\rm to},u^{\rm inc})$ satisfying 
\begin{equation*}
\left\{\begin{aligned}
& (\nabla\cdot A\nabla + \kappa^{2}\rho)u^{\rm to}=0 ,\quad (\Delta+\kappa^{2})u^{\rm inc}=0 && \text{in $D$,} \\ 
& u^{\rm to}=u^{\rm inc} ,\quad \nu\cdot A\nabla u^{\rm to} = \partial_{\nu}u^{\rm inc} && \text{on $\partial D$,} 
\end{aligned}\right. 
\end{equation*}
therefore the function $u^{\rm sc}:=u^{\rm to}-u^{\rm inc}$ verifies \eqref{eq:scattered-field}, i.e., the anisotropic medium $(D,A,\rho)$ is nonscattering with respect to such incident field $u^{\rm inc}$. This example does not contradict \Cref{thm:1} since the contrast $h$ defined in \eqref{eq:contrast-h} fails to satisfy the non-degeneracy condition \eqref{eq:non-degeneracy} at $x_{0}\in\partial D$. 
\end{example}

An analogous result for domains in $\mR^{n}$ with $n\ge 3$ that contain edge singularities is the following.

\begin{theorem}\label{thm:2}
Let $D$ be a bounded Lipschitz domain in $\mR^{n}$ with $n\ge 3$ such that $\mR^{n}\setminus\overline{D}$ is connected. Let $\kappa>0$, let $\rho\in L^{\infty}(D)$ be a positive real-valued function, let $A\in (C^{1}
(\overline{D}))^{n\times n}$ be a real symmetric matrix-valued function, satisfying the condition of uniform ellipticity \eqref{eq:ellipticity}. Assume that there exists $x_{0}\in\partial D$ such that conditions \ref{itm:nondegeneracy-contrast} and \ref{itm:strong-decay1} in \Cref{thm:1} are satisfied. If $\partial D$ contains an edge point $x_{0}\in\partial D$, then the anisotropic medium $(D,A,\rho)$ scatters every incident wave nontrivially. 
\end{theorem}

As in \Cref{thm:1}, it seems plausible 
to  relax the decay conditions to those in \eqref{Dini}.

The assumptions in our main theorems ensure the following degeneracy 
\begin{equation*} 
\nu\cdot A(x)\nabla u^{\rm sc}(x) = \nu\cdot(\id-A(x))\nabla u^{\rm inc}(x) \rightarrow 0 \quad \text{as $x\rightarrow x_{0}$}, 
\end{equation*} 
which was not addressed in our earlier work \cite{KSS23Anisotropic}. Here, we also allow $\nu\cdot(\id-A)\nabla u^{\rm inc}$ to change sign. In \cite{KSS23Anisotropic}, we focused instead on the scattering behavior under non-degeneracy conditions 
\begin{equation}
\left\{\begin{aligned}
& \nu\cdot(\id-A)\nabla u^{\rm inc}(x) \ge c > 0 \text{ for $\mH^{n-1}$-a.e. $x\in\partial D$ near $x_{0}$; or} \\ 
& \nu\cdot(\id-A)\nabla u^{\rm inc}(x) \le -c < 0 \text{ for $\mH^{n-1}$-a.e. $x\in\partial D$ near $x_{0}$.} 
\end{aligned}\right. \label{eq:nondegeneracy-condition}
\end{equation}
Although it is often possible to construct global solutions of the Helmholtz equation that are positive in a given set (see \cite{KSS23PositiveHelmholtz}), the real-valued functions $u^{\rm inc}$ and $\partial_j u^{\rm inc}$ typically exhibit numerous zeros, therefore the assumption \eqref{eq:nondegeneracy-condition} is not guaranteed to hold in many cases of interest. 

In view of condition \ref{itm:nondegeneracy-contrast}, it is natural to compare our results (\Cref{thm:1,thm:2}) with those of \cite{BPS14CornerScattering}, whose proof relies on suitable complex geometric optics (CGO) solutions; see also \cite{Bla18CornerScattering,BL17CornerScattering,BL21CornerScatteringSingleFarField,BL21Scattering,HSV16SingleMeasurement,PSV17CornerScattering,VX21nonscattering} for further refinements of similar results based on CGO or other harmonic-exponential solutions. This approach was extended in \cite{CX21CornerScatteringEllipticOperator} to scattering problems for general divergence-form operators, modeling \emph{isotropic} media rather than the Laplacian. We also refer to recent works \cite{HV25scattering1,VX25FinitenessResult} for further refinements in the case $A=\id$, based on the asymptotic behavior of suitable integral expressions. 
Our results (\Cref{thm:1,thm:2}) further extend these findings to general divergence-form operators modeling anisotropic media. 

Other approaches to this type of problem include \cite{EH15CornersEdgesScatter,EH18scattering}, which use direct expansions of solutions to the Helmholtz equation, and \cite{SS21NonscatteringFreeBoundary,KSS23Anisotropic}, which build on results related to free boundaries. It is also interesting to mention that the analysis in \cite{CakoniVogelius} is based on the observation that a nonscattering domain satisfies the assumptions of the \href{https://www.scilag.net/problem/G-180522.1}{Pompeiu conjecture} \cite{Pom29PompeiuProblem}, and hence its boundary must be analytic \cite{Wil76PompeiuProblem,Wil81PompeiuProblem}. Although this may be beyond our scope, we note that the \href{https://www.scilag.net/problem/G-180522.1}{Pompeiu conjecture} admits an equivalent formulation in terms of $k$-quadrature domains \cite{KLSS22QuadratureDomain}.


\subsubsection{Nonvanishing Bernoulli condition\label{sec:blowup-Bernoulli}}

The assumption \ref{itm:strong-decay1} in \Cref{thm:1,thm:2} enables us to establish our results by following the approach of \cite{SS25vanishingcontrast}. A natural question, however, is what happens when the condition \ref{itm:strong-decay1} in \Cref{thm:1} fails. We address this in the following theorem.

\begin{theorem}\label{thm:3}
Let $D\subset\mR^{2}$ be a bounded, simply connected open set with piecewise $C^{1}$ boundary. Let $\kappa>0$, let $\rho\in L^{\infty}(D)$ be a positive real-valued function, let $A\in (C^{1}
(\overline{D}))^{n\times n}$ be a real symmetric matrix-valued function, satisfying the condition of uniform ellipticity \eqref{eq:ellipticity}. Suppose that there exists $x_{0}\in\partial D$ such that 
\begin{enumerate}
\renewcommand{\labelenumi}{\theenumi}
\renewcommand{\theenumi}{(\roman{enumi})}
\item \label{itm:degeneracy-contrast} the contrast $h$ given in \eqref{eq:contrast-h} satisfies $\abs{h(x)}\le C\abs{x-x_{0}}^{\alpha}$, and  
\item \label{itm:non-decay1} $A(x)-\id = c_{0}B^{-1} + \tilde{R}(x-x_{0})$ for some real orthogonal matrix $B$ and for some nonzero constant $c_{0}$ with $\abs{\tilde{R}} \le C\abs{x}^{\alpha}$ and $\abs{\nabla \tilde{R}} \le C\abs{x}^{\alpha-1}$. 
\end{enumerate}
If $\partial D$ has a corner at $x_{0}$ with angle $\theta_{0}\notin \pi\mQ$, then the anisotropic medium $(D,A,\rho)$ scatters every incident wave $u^{\rm inc}\not\equiv 0$ nontrivially.  
\end{theorem}

\begin{remark}\label{rem:thm:3}
One can expand $u^{\rm inc}(Bx) - u^{\rm inc}(0) = H(x) + R(x)$, where $H\not\equiv 0$ is a harmonic homogeneous polynomial of order $m\ge 1$ and $\abs{R(x)}\le C\abs{x}^{m+1}$. The results of \Cref{thm:3} remain valid for any angle $\theta_{0}$ such that $\sin(m\theta_{0})\neq 0$. We also explain the difficulties of replacing $B^{-1}$ by a more general matrix in \Cref{rem:explaination-why-orthogonal}.
\end{remark}

Adapting Federer's dimension reduction argument, as in the proof of \cite[Theorem~1.10]{SS25vanishingcontrast} (compare \cite[Lemma~10.9]{Velichkov23FB} and \cite{Weiss99FB}), we can conclude an analogous result for domains in $\mR^{3}$ with $n\ge 3$: 

\begin{theorem}\label{thm:4} 
Let $D$ be a bounded Lipschitz domain in $\mR^{n}$ with $n\ge 3$ such that $\mR^{n}\setminus\overline{D}$ is connected. Let $\kappa>0$, let $\rho\in L^{\infty}(D)$ be a positive real-valued function, let $A\in (C^{1}
(\overline{D}))^{n\times n}$ be a real symmetric matrix-valued function, satisfying the condition of uniform ellipticity \eqref{eq:ellipticity}. If $x_{0}\in\partial D$ is an edge point with angle $\theta_{0}\notin \pi\mQ$, and conditions \ref{itm:degeneracy-contrast} and \ref{itm:non-decay1} in \Cref{thm:3} hold, then the anisotropic medium $(D,A,\rho)$ scatters every incident wave $u^{\rm inc}\not\equiv 0$ nontrivially. 
\end{theorem}

\begin{remark}
Let $m$ be the integer appearing in  \Cref{rem:thm:3}. The results of \Cref{thm:4} remain valid for any angle $\theta_{0}$ such that $\sin(k\theta_{0})\neq 0$ for all $k=1,\cdots,m$. 
\end{remark}

In our previous work \cite[Example~1.5]{KSS23Anisotropic}, we showed that nonscattering domains may have corners with opening angles $\ell\pi/m$ for integers $m\ge 2$ and $1\le \ell < 2m-1$, highlighting the necessity of the condition $\theta_{0}\notin \pi\mQ$. See also \cite{CHLX24examples} for additional examples of nonscattering inhomogeneities. 
It is also worth mentioning that \cite{CDLZ21VanishingOrder,CDLZ21VanishingOrder2} studied an inverse problem in $\mR^{3}$: 
\begin{equation*}
\left\{\begin{aligned}
& (\Delta + \kappa^{2})u^{\rm to}=0 && \text{in $\mR^{3}\setminus\overline{D}$,} \\ 
& u^{\rm to} = u^{\rm inc} + u^{\rm sc} && \text{in $\mR^{3}\setminus\overline{D}$,} \\ 
& \partial_{\nu}u^{\rm to} + \eta u^{\rm to} = 0 && \text{on $\partial D$,} \\ 
& \text{$u^{\rm sc}$ is outgoing,}
\end{aligned}\right.
\end{equation*} 
obtaining results for polyhedral domains $D$ with vertices having angles $\theta_{0}\notin \pi\mQ$. Although their analysis was restricted to three dimensions, the results can be readily extended to arbitrary dimensions $d\ge 2$ using $d$-dimensional spherical harmonics \cite{EF14SphericalHarmonics}. They proceed by examining the vanishing order of general solutions to the Helmholtz equation, which they refer to as ``generalized Laplacian eigenfunctions''. This approach relies on explicit expansions of Helmholtz solutions and is therefore highly specific to the Laplace operator. It is worth noting that the analysis in \cite{HV25scattering2} also relies on the observation that a nonscattering domain satisfies the assumptions of the \href{https://www.scilag.net/problem/G-180522.1}{Pompeiu conjecture} \cite{Pom29PompeiuProblem}, incorporating the moving plane technique from \cite{BK82Pompeiu}, and is therefore highly specific to the Laplace operator. 
In contrast, our findings show that blowup analysis can be naturally extended to general elliptic operators beyond the Laplacian.

\section{Proof of \texorpdfstring{\Cref{thm:1,thm:2}}{Theorems \ref{thm:1} and \ref{thm:2}}} \label{sec_2}

Suppose that \eqref{eq:scattered-field} holds. For each $\psi\in C_{c}^{\infty}(\mR^{n})$, one computes that (we slightly abuse some notations here) 
\begin{equation*}
\begin{aligned} 
& \int_{\mR^{n}} \psi \left(-(\nabla\cdot (A(x)-\id)\nabla+\kappa^{2}(\rho(x)-1))u^{\rm inc}\right)\mL^{n}\lfloor D \,\rmd x \\ 
& \quad = \int_{D} \psi \left(-(\nabla\cdot (A(x)-\id)\nabla+\kappa^{2}(\rho(x)-1))u^{\rm inc}\right) \, \rmd x \\ 
& \quad = \int_{D} \psi (\nabla\cdot A(x)\nabla+\kappa^{2}\rho(x))u^{\rm sc} \,\rmd x \\ 
& \quad = \int_{D} \left(-A(x)\nabla\psi \cdot \nabla u^{\rm sc} +\kappa^{2}\rho(x)u^{\rm sc}\right) \,\rmd x + \int_{\partial D} \psi \nu\cdot A(x)\nabla u^{\rm sc} \, \rmd S \\ 
& \quad = \int_{\mR^{n}} \left(-A(x)\nabla\psi \cdot \nabla u^{\rm sc} + \kappa^{2}\rho(x)u^{\rm sc}\right) \,\rmd x + \int_{\partial D} \psi \left(\nu\cdot(\id-A)\nabla u^{\rm inc}\right) \, \rmd S \\ 
& \quad = \int_{\mR^{n}} \psi (\nabla\cdot A(x)\nabla+\kappa^{2}\rho(x))u^{\rm sc} \,\rmd x + \int_{\mR^{n}} \psi \left(\nu\cdot(\id-A)\nabla u^{\rm inc}\right)\mH^{n-1}\lfloor\partial D \, \rmd x.  
\end{aligned} 
\end{equation*}
This shows that \eqref{eq:scattered-field} is equivalent to 
\begin{subequations}\label{eq:scattered-field-distribution-form} 
\begin{equation}
(\nabla\cdot A(x)\nabla+\kappa^{2}\rho(x))u^{\rm sc} = f \mL^{n}\lfloor D + g \mH^{n-1}\lfloor\partial D \text{ in $\mR^n$},\quad u^{\rm sc}|_{\mR^{n}\setminus\overline{D}}=0 \label{eq:scattered-field-distribution-form-a} 
\end{equation}
in distribution sense, where 
\begin{equation}
f = -(\nabla\cdot (A(x)-\id)\nabla+h)u^{\rm inc} ,\quad g = \nu\cdot(A-\id)\nabla u^{\rm inc} \label{eq:scattered-field-distribution-form-b} 
\end{equation}
and $h(x) := \kappa^{2}(\rho(x)-1)\chi_{D}$. 
\end{subequations}  

The above observations allow us to establish \Cref{thm:1,thm:2} by following the approach of \cite{SS25vanishingcontrast}. 
In \Cref{sec:C01-regularity} we prove the $C^{0,1}$ regularity and decay rates for solutions of a certain class of PDE. In \Cref{sec:blowup}, we study blowup limits and prove their homogeneity by using a modified balanced energy functional, and also recall some results related to blowup solutions. In \Cref{sec:weak-flatness}, we show non-degeneracy and weak flatness properties and regularity of the free boundary. Finally, we prove \Cref{thm:1,thm:2} in \Cref{sec:proof}.

\subsection{Lipschitz regularity and optimal decay rate\label{sec:C01-regularity}} 


In this section we study the regularity and vanishing order for solutions of the equation 
\begin{equation}
(\nabla\cdot A(x)\nabla+q)u=f\mL^{n}\lfloor B_{2}+g\mH^{n-1}\lfloor\partial D \text{ in $B_{2}$} ,\quad u|_{B_{2}\setminus\overline{D}}=0, \label{eq:PDE1} 
\end{equation}
as follows: 

\begin{lemma}\label{lem:regularity1}
Let $D$ be a bounded Lipschitz domain in $\mR^{n}$ with $0\in\partial D$, let $A\in(C^{0,1}(\overline{D}))^{n\times n}$ be a real symmetric matrix-valued function, satisfying the condition of uniform ellipticity \eqref{eq:ellipticity}, let $q\in L^{\infty}(B_{2})$, let $m\ge 0$ be an integer and suppose that $u\in H^{1}(B_{2})$ solves \eqref{eq:PDE1} with $\abs{f(x)}\le C\abs{x}^{m}$ a.e. in $B_{2}$ and $\abs{g(x)}\le C\abs{x}^{m+1}$ for $\mH^{n-1}$-a.e. $x\in\partial D\cap B_{2}$. Then 
\begin{equation*}
\abs{u(x)} + \abs{x} \abs{\nabla u(x)} \le C\abs{x}^{m+2} \quad \text{in $B_{1}$.} 
\end{equation*}
\end{lemma}

\begin{remark*} 
We will show in \Cref{lem:nondegeneracy} below that this decay rate is optimal. 
\end{remark*}

\begin{proof}[Proof of \Cref{lem:regularity1}] 
First, by \cite[Lemma~2.2]{KSS23Anisotropic}, we obtain that $u\in C^{0.1}(B_{1})$ and we see that $S_{r}:=\norm{u}_{L^{\infty}(B_{r})} + r\norm{\nabla u}_{L^{\infty}(B_{r})}$ is monotone non-decreasing, then the set $\mD$ of points, at which it is discontinuous, is at most infinitely countable. 
We consider the set $\phi(\mD)$, where $\phi(r)=r^{-m-2}S_{r}$, which is at most infinitely countable. We want to prove that $S_{r}\le Cr^{m+2}$. Suppose the contrary: for any increasing positive sequence $\{\rho_{j}\}$ with $\{\rho_{j}\}\cap\phi(\mD)=\emptyset$ and $\rho_{j}\rightarrow\infty$ as $j\rightarrow\infty$, there is $\tilde{r}_{j}\in(0,1]$ such that $S_{\tilde{r}_{j}}>\rho_{j}\tilde{r}_{j}^{m+2}$. Choosing $r_{j}\in(0,1]$ to be the supremum of all $\tilde{r}_{j}\in(0,1]$ with this property, we have 
\begin{equation*}
S_{r_{j}} = \rho_{j}r_{j}^{m+2} ,\quad S_{r} \le \rho_{j}r^{m+2} \quad \text{for $r\ge r_{j}$,}
\end{equation*} 
and the sequence $\{r_{j}\}$ is nonincreasing with $r_{j}\rightarrow 0$ as $j\rightarrow\infty$. 

If we define the rescaled functions 
\begin{equation*}
u_{j}(x) := \frac{u(2r_{j}x)}{\rho_{j}r_{j}^{m+2}} \quad \text{for all $x\in B_{1/r_{j}}$,} 
\end{equation*}
then 
\begin{equation}
\norm{u_{j}}_{L^{\infty}(B_{1/2})} + \norm{\nabla u_{j}}_{L^{\infty}(B_{1/2})} = \frac{S_{r_{j}}}{\rho_{j}r_{j}^{m+2}} = 1 \label{eq:LB-weak-star}
\end{equation}
and 
\begin{equation*}
\norm{u_{j}}_{L^{\infty}(B_{1})} + \norm{\nabla u_{j}}_{L^{\infty}(B_{1})} = \frac{S_{2r_{j}}}{\rho_{j}r_{j}^{m+2}} \le \frac{\rho_{j}(2r_{j})^{m+2}}{\rho_{j}r_{j}^{m+2}} = 2^{m+2}. 
\end{equation*}
Since $(u_{j})$ is uniformly bounded in $C^{0,1}(\overline{B_{1/2}})$, and by Banach-Alaoglu theorem there is a subsequence\footnote{By Arzela-Ascoli theorem, the sequence also converges strongly to $v$ in $C^{0}(\overline{B_{1/2}})$.}, still denoted by $(u_{j})$, converging to some $v$ in $C^{0,1}(\overline{B_{1/2}})$ weak-$\star$. 

On the other hand, 
\begin{equation*}
\nabla_{x}\cdot(A(2r_{j}x)\nabla u_{j}(x)) + \frac{4}{\rho_{j}}(qu)(2r_{j}x) = \frac{4}{\rho_{j}}f(2r_{j}x)\mL^{n}\lfloor B_{2} + \frac{4}{\rho_{j}r_{j}} g(2r_{j}x) \mH^{n-1} \lfloor \partial D 
\end{equation*}
in $B_{1}$, more precisely, 
\begin{equation*}
\begin{aligned}
& -\int_{B_{1}} A(2r_{j}x)\nabla u_{j}(x) \cdot \nabla\phi(x)\,\rmd x \\
& \quad = \frac{4}{\rho_{j}}\int_{B_{1}} f(2r_{j}x)\phi(x)\,\rmd x + \frac{4}{\rho_{j}r_{j}} \int_{\partial D} g(2r_{j}x)\phi(x)\,\rmd S_{x} - \frac{4}{\rho_{j}}\int_{B_{1}}(qu)(2r_{j}x)\phi(x)\,\rmd x 
\end{aligned}
\end{equation*}
for all $\phi\in C_{c}^{\infty}(B_{1})$. It is not difficult to see that 
\begin{equation*}
\begin{aligned}
& \left| \frac{4}{\rho_{j}}\int_{B_{1}} f(2r_{j}x)\phi(x)\,\rmd x + \frac{4}{\rho_{j}r_{j}} \int_{\partial D} g(2r_{j}x)\phi(x)\,\rmd S_{x} \right| \le C\norm{\phi}_{L^{\infty}(B_{1})} \rho_{j}^{-1} 
\end{aligned}
\end{equation*}
and 
\begin{equation*}
\left| \frac{4}{\rho_{j}}\int_{B_{1}}(qu)(2r_{j}x)\phi(x)\,\rmd x \right| \le C \norm{\phi}_{L^{\infty}(B_{1})} r_{j}^{2}, 
\end{equation*}
therefore 
\begin{equation}
\begin{aligned}
& \left| \int_{B_{1}} A(0)\nabla u_{j}(x) \cdot \nabla\phi(x)\,\rmd x \right| \\
& \quad \le \left| \int_{B_{1}} (A(2r_{j}x)-A(0))\nabla u_{j}(x) \cdot \nabla\phi(x)\,\rmd x \right|  + \left| \int_{B_{1}} A(2r_{j}x)\nabla u_{j}(x) \cdot \nabla\phi(x)\,\rmd x \right| \\
& \quad \le C\norm{\nabla\phi}_{L^{\infty}(B_{1})}r_{j} + C\norm{\phi}_{L^{\infty}(B_{1})}(\rho_{j}^{-1} + r_{j}^{2}). 
\end{aligned} \label{eq:PDE-weak-star} 
\end{equation}

Applying the weak-$\star$ convergence of $(u_{j})$ to $v$ in \eqref{eq:LB-weak-star} and \eqref{eq:PDE-weak-star}, we obtain 
\begin{equation*}
\norm{v}_{L^{\infty}(B_{1/2})} + \norm{\nabla v}_{L^{\infty}(B_{1/2})} \ge 1 ,\quad \nabla\cdot A(0)\nabla v=0 \text{ in $B_{1}$ (distribution sense).}
\end{equation*}
Moreover, since $D$ is a Lipschitz domain with $0\in \partial D$ and $u|_{B_{2}\setminus\overline{D}}=0$, it follows that there is an open cone $C$ in $\mR^{n}$ so that each $u_{j}$ and hence $v$ vanish in $C\cap B_{1/2}$. By the unique continuation principle for the elliptic operator $\nabla\cdot A(0)\nabla$, we conclude $v=0$ in $B_{1/2}$. This contradicts the condition $\norm{v}_{L^{\infty}(B_{1/2})} + \norm{\nabla v}_{L^{\infty}(B_{1/2})} \ge 1$. 
\end{proof}

\subsection{Blowup solutions\label{sec:blowup}} 

Throughout this section, we will make the following standing assumptions: Let $D$ be a bounded Lipschitz domain with $0\in\partial D$, let $q\in L^{\infty}(B_{2})$ and let $u\in C_{\rm loc}^{0,1}(B_{2})$ be a solution to  
\begin{subequations} \label{eq:blowup-u}
\begin{equation}
(\nabla\cdot A(x)\nabla + q)u = f \mL^{n}\lfloor \{u\neq 0\} + g \mH^{n-1}\lfloor \partial D, 
\end{equation}
satisfying 
\begin{equation}
\abs{u(x)} + \abs{x} \abs{\nabla u(x)} \le C\abs{x}^{m+2}. \label{eq:C01-regularity}
\end{equation}
Note that the function $u_{r}(x):=u(rx)/r^{m+2}$ also satisfies the estimate \eqref{eq:C01-regularity} and solves the equation 
\begin{equation*}
\begin{aligned} 
\nabla\cdot A(rx)\nabla u_{r}(x) &= (r^{-m}f(rx) - r^{2}q(rx)u_{r}(x)) \mL^{n}\lfloor \{u_{r}\neq 0\} \\
& \quad + r^{-m-1} (g\mH^{n-1}\lfloor\partial D)(rx). 
\end{aligned} 
\end{equation*} 
Similar to \cite[Lemma 2.5]{SS25vanishingcontrast}, we now assume that 
\begin{equation}
\begin{aligned}
& \text{$f = H+R$, where $H$ is homogeneous polynomial of degree $m$} \\ 
& \text{and $\abs{R(x)}\le C\abs{x}^{m+\alpha}$ for some $\alpha>0$,}  
\end{aligned}\label{eq:blowup-assu1}
\end{equation} 
as well as 
\begin{equation}
\text{$\abs{g(x)} \le C\abs{x}^{m+1+\alpha}$ for $\mH^{n-1}$-a.e. $x\in\partial D\cap B_{2}$.}  \label{eq:blowup-assu2}
\end{equation} 
We also assume that $A\in (C^{0,1}(\overline{D}))^{n\times n}$ is a real symmetric matrix-valued function, satisfying the condition of uniform ellipticity \eqref{eq:ellipticity}, as well as 
\begin{equation}
\abs{A(x)-\id}\le C\abs{x} ,\quad \abs{\nabla A(x)}\le C. \label{eq:A-approx-id}
\end{equation}
\end{subequations}
In view of Banach-Alaoglu theorem, we say that $v$ is a \emph{blowup limit of $u$ of order $m+2$ at $0$} if there is a sequence $r_{j}\rightarrow 0$ so that $u_{r_{j}}\rightarrow v$ in $C^{0,1}(\overline{B_{1}})$ weak-$\star$.

Similarly as in \cite[Section~3.1]{SS25vanishingcontrast}, we introduce a balanced  energy functional (see \cite[Lemma~22]{CSY18FB} or \cite[Section~3.5]{PSU12FreeBoundary}): 
\begin{equation*}
\begin{aligned} 
W_{A}(r,u) &:= \frac{1}{r^{2m+n+2}} \int_{B_{r}} \left( A(x)\nabla u\cdot\nabla u + 2Hu \right) \, \rmd x - \frac{m+2}{r^{2m+n+3}} \int_{\partial B_{r}} u^{2} \, \rmd S \\
& = \int_{B_{1}} \left(A(rx)\nabla u_{r}\cdot\nabla u_{r} + 2Hu_{r}\right) \,\rmd x - (m+2) \int_{\partial B_{1}} u_{r}^{2} \,\rmd S. 
\end{aligned}
\end{equation*}
We compute its derivative (we slightly abuse the notations in the computations below) 
\begin{equation*}
\begin{aligned}
\frac{1}{2}\partial_{r}W_{A}(r,u) &= \int_{B_{1}} \left( A(rx)\nabla u_{r}\cdot \nabla \partial_{r}u_{r} + H \partial_{r}u_{r} \right) \,\rmd x + \frac{1}{2}\int_{B_{1}} \partial_{r}(A(rx)) \nabla u_{r}\cdot\nabla u_{r} \, \rmd x \\ 
& \quad - (m+2) \int_{\partial B_{1}} u_{r} \partial_{r}u_{r} \,\rmd S \\ 
&= \int_{B_{1}} \left( -\nabla\cdot(A(rx)\nabla u_{r}) + H \right) \partial_{r}u_{r} \,\rmd x + \frac{1}{2}\int_{B_{1}} \partial_{r}(A(rx)) \nabla u_{r}\cdot\nabla u_{r} \, \rmd x \\ 
& \quad + \int_{\partial B_{1}} \overbrace{\left( \partial_{\nu}u_{r} - (m+2)u_{r} \right)}^{=\,r\partial_{r}u_{r}} \partial_{r}u_{r} \, \rmd S+ \int_{\partial B_{1}} \nu\cdot (A(rx)-\id) \nabla u_{r} \partial_{r}u_{r} \,\rmd S \\ 
&= \int_{\partial B_{1}} r(\partial_{r}u_{r})^{2}\,\rmd S -\int_{B_{1}} r^{-m}R(rx)\partial_{r}u_{r}\,\rmd x + \int_{B_{1}} r^{2}q(rx)u_{r}\partial_{r}u_{r}\,\rmd x \\
& \quad + \frac{1}{2}\int_{B_{1}} \partial_{r}(A(rx)) \nabla u_{r}\cdot\nabla u_{r} \, \rmd x \\ 
& \quad  - r^{-m-1}\int_{B_{1}\cap (r^{-1}\partial D)} g(rx)\partial_{r}u_{r} \,\rmd S + \int_{\partial B_{1}} \nu\cdot (A(rx)-\id) \nabla u_{r} \partial_{r}u_{r} \,\rmd S .
\end{aligned}
\end{equation*} 
We introduce the functional 
\begin{equation*}
\begin{aligned}
F_{A}(r,u) &= 2\int_{0}^{r}\int_{B_{1}} \tau^{-m}R(\tau x) \partial_{\tau}u_{\tau}\,\rmd x\,\rmd\tau - 2 \int_{0}^{r}\int_{B_{1}} \tau^{2}q(\tau x)u_{\tau}\partial_{\tau}u_{\tau}\,\rmd x\,\rmd \tau \\ 
& \quad + \int_{0}^{r}\int_{B_{1}} \partial_{\tau}(A(\tau x))\nabla u_{\tau}\cdot\nabla u_{\tau} \,\rmd x \,\rmd\tau \\ 
& \quad + 2\int_{0}^{r} \int_{B_{1}\cap(\tau^{-1}\partial D)} \tau^{-m-1}g(\tau x) \partial_{\tau}u_{\tau} \,\rmd S\,\rmd\tau \\
& \quad - 2 \int_{0}^{r} \int_{\partial B_{1}} \nu\cdot (A(\tau x)-\id) \nabla u_{\tau} \partial_{\tau}u_{\tau} \,\rmd S \,\rmd\tau. 
\end{aligned}
\end{equation*}
in order to extract the following non-negative quantity: 
\begin{equation*}
\partial_{r}\left(W_{A}(r,u)-F_{A}(r,u)\right) = 2 \int_{\partial B_{1}} r(\partial_{r}u_{r})^{2} \,\rmd S \ge 0. 
\end{equation*}

We denote 
\begin{equation*}
W(r,u) := W_{\id}(r,u) = \frac{1}{r^{2m+n+2}} \int_{B_{r}} \left( \abs{\nabla u}^{2} + 2Hu \right) \, \rmd x - \frac{m+2}{r^{2m+n+3}} \int_{\partial B_{r}} u^{2} \, \rmd S. 
\end{equation*}
From \eqref{eq:C01-regularity} and \eqref{eq:A-approx-id} we see that 
\begin{equation}
\abs{W(r,u)-W_{A}(r,u)} = \frac{1}{r^{2m+n+2}} \left| \int_{B_{r}} (A(x)-\id)\nabla u\cdot \nabla u \,\rmd x \right| \le Cr. \label{eq:approx-Weiss-functional}
\end{equation} 
Since $u_{r}$ satisfies the estimates \eqref{eq:C01-regularity} and 
\begin{equation*}
\abs{\partial_{r}u_{r}(x)} = \abs{-(m+2)r^{-m-3}u(rx)+r^{-m-2}x\cdot\nabla u(sx)} \le Cr^{-1}\abs{x}^{m+2}, 
\end{equation*}
then 
\begin{equation*}
\begin{aligned}
\abs{F_{A}(r,u)} &\le C \int_{0}^{r} \int_{B_{1}} (\tau^{\alpha}\abs{x}^{m+\alpha}+\tau^{2}\abs{u_{\tau}(x)}) \abs{\partial_{\tau}u_{\tau}(x)} \,\rmd x\,\rmd\tau \\ 
& \quad + C\int_{0}^{r}\int_{B_{1}} \tau \abs{\nabla u_{\tau}}^{2}\,\rmd x\,\rmd\tau + C\int_{0}^{r} \int_{B_{1}\cap(\tau^{-1}\partial D)} \tau^{\alpha}\abs{x}^{m+1+\alpha} \abs{\partial_{\tau}u_{\tau}(x)} \,\rmd S \,\rmd\tau \\ 
& \quad + C \int_{0}^{r} \int_{\partial B_{1}} \tau\abs{x} \abs{\nabla u_{\tau}} \abs{\partial_{\tau}u_{\tau}} \,\rmd S \,\rmd\tau \\ 
& \le C r^{\alpha}
\end{aligned}
\end{equation*}
It is easy to see that  
\begin{equation*}
W_{A}(r,u) \ge -C. 
\end{equation*}
Thus, the non-decreasing quantity $r\mapsto W_{A}(r,u)+F_{A}(r,u)$ has a finite limit as $r\rightarrow 0$, and this limit equals $W_{A}(0+,u)$. Now by \eqref{eq:approx-Weiss-functional} we see that the quantity $r\mapsto W(r,u)$ has a finite limit as $r\rightarrow 0$, and this limit equals $W(0+,u)=W_{A}(0+,u)$. Since 
\begin{equation*}
W(rs,u) = W(s,u_{r}) \quad \text{for all $r,s\in(0,1]$,} 
\end{equation*}
we can use the standard argument in \cite[Lemma~16]{Yeressian2016} (we omit the details here) to conclude the following lemma: 

\begin{lemma}\label{lem:blowup-procedure}
Suppose $u\in C_{\rm loc}^{0,1}(B_{2})$ satisfies \eqref{eq:blowup-u}. Any blowup limit $v$ (of order $m+2$) of $u$ is homogeneous of degree $m+2$ solves the equation 
\begin{equation*}
\Delta v = H\mL^{n}\lfloor\{v\neq 0\} \quad \text{in $B_{1}$.} 
\end{equation*}
One has $W(s,v)=W(0+,u)$ for $0<s\le 1$. Moreover, the homogeneous degree $m+2$ extension of $v$ (still denoted by $v$) solves the equation 
\begin{equation}
\Delta v = H\mL^{n}\lfloor\{v\neq 0\} \quad \text{in $\mR^{n}$.} \label{eq:blowup-solution}
\end{equation}
\end{lemma}



We recall the following result, which characterizes all blowup limits in two dimensions: 

\begin{lemma}[{\cite[Theorem~3.12]{SS25vanishingcontrast}}]\label{lem:SS25-lem3.12}
Let $n=2$. Suppose that $v\in C_{\rm loc}^{1,1}(\mR^{2})$ is nontrivial, homogeneous of degree $m+2$ and solves the equation 
\begin{equation*}
\Delta v = H \mL^{n}\lfloor \{v\neq 0\} \quad \text{in $\mR^{2}$} 
\end{equation*}
where $H$ is a harmonic homogeneous polynomial of degree $m$. Then one of the following holds: 
\begin{enumerate}
\renewcommand{\labelenumi}{\theenumi}
\renewcommand{\theenumi}{\rm (\alph{enumi})} 
\item $\supp\,(v)$ is a half space and $v$ is a polynomial\footnote{After a rotation, the form given in \cite[Lemma~3.3 or Lemma~3.5]{SS25vanishingcontrast}.} of order $m+2$. 
\item $\supp\,(v)=\mR^{2}$ and $v=\frac{1}{4m+4}\abs{x}^{2}H + w$ where $w$ is a harmonic homogeneous polynomial of degree $m+2$. 
\end{enumerate}
\end{lemma}

\subsection{Non-degeneracy and weak flatness\label{sec:weak-flatness}} 

Throughout this section, we will make the following standing assumptions: Let $D$ be a bounded Lipschitz domain with $0\in\partial D$, let $q\in C_{\rm loc}^{\alpha}(B_{2})$ and let $u\in C_{\rm loc}^{0,1}(B_{2})$ be a solution to 
\begin{subequations} \label{eq:blowup-u1}
\begin{equation}
(\nabla\cdot A(x)\nabla + q)u = f \mL^{n}\lfloor D + g \mH^{n-1}\lfloor \partial D ,\quad u|_{B_{2}\setminus\overline{D}}=0.  
\end{equation}
We assume that 
\begin{equation}
\begin{aligned}
&\text{$f=P+R$, where $P\not\equiv 0$ is a homogeneous polynomial of order $m\ge 0$} \\ 
&\text{and $\abs{R(x)}\le C\abs{x}^{m+\alpha}$ for some $\alpha>0$}
\end{aligned}
\end{equation}
\end{subequations} 
We also assume that $g$ satisfies \eqref{eq:blowup-assu2} and $A$ satisfies \eqref{eq:A-approx-id}. 
In this setting, we say that $v$ is a blowup limit of $u$ at 0 if there is a sequence $r_{j}\rightarrow 0$ so that $u_{r_{j}}\rightarrow v$ in $C^{0,1}(\overline{B_{1}})$ weak-$\star$, where $u_{r}(x):=u(rx)/r^{m+2}$.

On the other hand, since $D$ is a Lipschitz domain, there are $\delta>0$ and $r_{0}>0$ so that we have the following consequence of the interior cone property: 
\begin{equation}
\text{If $x\in\overline{D}\cap\overline{B_{r_{0}}}$, then $B_{r}(x)\cap\overline{D}$ contains a ball of radius $\delta r$ whenever $r\le r_{0}$.} \label{eq:interior-cone}
\end{equation} 
Based on the above observations, as in \cite[Lemma~4.1]{SS25vanishingcontrast}, we first prove a non-degeneracy result showing that $u$ cannot vanish faster than a  certain order in $\supp\,(u)$: 

\begin{lemma}\label{lem:nondegeneracy}
Suppose $u\in C_{\rm loc}^{0,1}(B_{2})$ satisfies \eqref{eq:blowup-u1}. For any $\epsilon\in(0,1)$, there exists a pair of positive numbers $(r_{\epsilon},c_{\epsilon})$ such that 
\begin{equation*}
\norm{u}_{L^{\infty}(B_{\epsilon\abs{x}}(x))} \ge c_{\epsilon} \abs{x}^{m+2} \quad \text{for all $x\in\overline{D}\cap\overline{B_{r_{\epsilon}}}$}. 
\end{equation*}
\end{lemma}

\begin{remark*}
If we further assume that 
\begin{equation}
\text{$\overline{B_{r_{0}}}\cap\overline{D}\subset\supp\,(f)$ for some $r_{0}>0$,} \label{eq:observation-supp-u}
\end{equation}
then \Cref{lem:nondegeneracy} remains valid with $r_{\epsilon}$ replaced by a constant $r_{0}$ that is independent of $\epsilon$. The condition \eqref{eq:observation-supp-u} was valid in the setting of \cite{SS25vanishingcontrast}. However, the choice made in \eqref{eq:scattered-field-distribution-form-b} does not guarantee \eqref{eq:observation-supp-u}, unless the coefficient $A$ is analytic in a neighborhood of $x_{0}=0$. Thus the statement and proof of \Cref{lem:nondegeneracy} are slightly different from \cite[Lemma~4.1]{SS25vanishingcontrast}.
\end{remark*}

\begin{proof}[Proof of \Cref{lem:nondegeneracy}]
Suppose, to the contrary, that there exists some $\epsilon \in (0,1)$ for which no such pair of positive numbers $(r_{\epsilon},c_{\epsilon})$ can be found. That is, for any pair of positive numbers $(r,c)$, there exists $x\in\overline{D}\cap\overline{B_{r}}$ such that  
\begin{equation*}
\norm{u}_{L^{\infty}(B_{\epsilon\abs{x}}(x))} < c\abs{x}^{m+2}. 
\end{equation*}
For each $j\ge 1$, by choosing the pair $(r,c)=(1/j,1/j)$, there exists $x_{j}\in\overline{D}\cap\overline{B_{1/j}}$ such that 
\begin{equation}
\norm{u}_{L^{\infty}(B_{\epsilon\abs{x_{j}}}(x_{j}))} < \frac{1}{j}\abs{x_{j}}^{m+2}. \label{eq:contradiction-nondegenerate}
\end{equation}
There exists a subsequence, still denoted as $(x_{j})$, such that $x_{j}\rightarrow 0$. Let $r_{j}=\abs{x_{j}}$ and the blowup sequence $u_{r_{j}}(x):=u(r_{j}x)/r_{j}^{m+2}$. Now using \eqref{eq:interior-cone} with $x=x_{j}$ and $r=\epsilon r_{j}$, for each sufficiently large $j\ge 1$ there is $y_{j}$ so that 
\begin{equation}
B_{\delta\epsilon r_{j}}(y_{j}) \subset B_{\epsilon r_{j}}(x_{j})\cap\overline{D}. \label{eq:interior-cone1}
\end{equation}
It also follows that $(1-\epsilon)r_{j}\le\abs{y_{j}}\le(1+\epsilon)r_{j}$. Then by \eqref{eq:contradiction-nondegenerate} we have 
\begin{equation*}
\norm{u}_{L^{\infty}(B_{\delta\epsilon r_{j}}(y_{j}))} < \frac{1}{j}r_{j}^{m+2}. 
\end{equation*}
Writing $z_{j}=y_{j}/r_{j}$, this means that 
\begin{equation*}
\norm{u_{r_{j}}}_{L^{\infty}(B_{\delta\epsilon}(z_{j}))} < \frac{1}{j}. 
\end{equation*}
By considering a suitable subsequence, we may assume that $z_{j}$ converges to some $z$ with $1-\epsilon\le\abs{z}\le1+\epsilon$. Thus we have 
\begin{equation}
u_{r_{j}}\rightarrow 0 \quad \text{in $L^{\infty}(B_{\delta\epsilon/2}(z))$}. \label{eq:contradiction-blowup}
\end{equation}
On the other hand, since 
\begin{equation*}
\begin{aligned} 
\nabla\cdot A(r_{j}x)\nabla u_{r_{j}}(x) &= - r_{j}^{2}q(r_{j}x)u_{r_{j}}(x) + r_{j}^{-m}(f \mL^{n}\lfloor D)(r_{j}x) \\
& \quad + r_{j}^{-m-1} (g\mH^{n-1}\lfloor\partial D)(r_{j}x), 
\end{aligned} 
\end{equation*}
and since $\abs{g(x)}\le C\abs{x}^{m+1+\alpha}$, then we have 
\begin{equation*}
\nabla\cdot A(r_{j}x)\nabla u_{r_{j}}(x) = - r_{j}^{2}q(r_{j}x)u_{r_{j}}(x) + r_{j}^{-m}(f \mL^{n}\lfloor D)(r_{j}x) + O(r_{j}^{\alpha}) (\mH^{n-1}\lfloor\partial D)(r_{j}x)
\end{equation*}
uniformly on $x\in\overline{B_{1}}$. Now \eqref{eq:interior-cone1} implies that 
\begin{equation*}
\nabla\cdot A(r_{j}x)\nabla u_{r_{j}}(x) = - r_{j}^{2}q(r_{j}x)u_{r_{j}}(x) + P(x) + r_{j}^{-m}R(r_{j}x) + O(r_{j}^{\alpha}) (\mH^{n-1}\lfloor\partial D)(r_{j}x) \text{ in $B_{\delta\epsilon}(z_{j})$}. 
\end{equation*}
Since $r_{j}\rightarrow 0$ and $\abs{u_{r_{j}}(x)}\le C$, $\abs{R(x)}\le C\abs{x}^{m+\alpha}$, we have 
\begin{equation*}
\nabla\cdot A(r_{j}x)\nabla u_{r_{j}} \rightarrow P \quad \text{in the distributional sense on $B_{\delta\epsilon/2}(z)$}. 
\end{equation*}
Since $P$ does not vanish on any open subset of $\calS^{n-1}$, this contradicts \eqref{eq:contradiction-blowup} by uniqueness of distributional limits, completing the proof by contradiction. 
\end{proof}

As shown in \Cref{lem:blowup-procedure}, the blowup limit $u_{0}$ in this work is identical to the one in \cite[Section~4]{SS25vanishingcontrast}. With \Cref{lem:nondegeneracy} at hand, we can establish some weak flatness properties for $\partial D$ at $0$, using the exact same argument in \cite[Lemma~4.3]{SS25vanishingcontrast}: 

\begin{lemma}\label{lem:SS25-lem4.3} 
Suppose $u\in C_{\rm loc}^{0,1}(B_{2})$ satisfies \eqref{eq:blowup-u1} and suppose that $u_{0}$ is a blowup limit of such $u$ at $0$ with $\supp\,(u_{0})=\overline{\mR_{+}^{n}}$. Then for any $\delta>0$ there is $r>0$ so that $\partial D\cap B_{r} \subset \{\abs{x_{n}} \le \delta r\}$. 
\end{lemma} 

\subsection{Conclusions\label{sec:proof}} 

To prepare for the main results, we summarize the preceding discussion in the following proposition.

\begin{proposition}[Properties of blowups: I]\label{prop:properties-blowups1}
Let $D\subset\mR^{n}$ be a Lipschitz domain such that $0\in\partial D$, let $q\in C_{\rm loc}^{\alpha}(B_{2})$, let $f\in C_{\rm loc}^{\alpha}(B_{2})$, let $A\in (C^{0,1}(\overline{D}))^{n\times n}$ be a real symmetric matrix-valued function, satisfying the condition of uniform ellipticity \eqref{eq:ellipticity} and 
\begin{equation*}
\abs{A(x)-\id}\le C\abs{x} ,\quad \abs{\nabla A(x)}\le C. 
\end{equation*}
Suppose that $u\in H_{\rm loc}^{1}(B_{2})$ solves 
\begin{equation*}
(\nabla\cdot A(x)\nabla + q)u = f \mL^{n}\lfloor D + g \mH^{n-1}\lfloor \partial D ,\quad u|_{B_{2}\setminus\overline{D}}=0.  
\end{equation*}
Let $m\ge 0$ be an integer. Assume that $\abs{g(x)}\le C\abs{x}^{m+1+\alpha}$ for $\mH^{n-1}$-a.e. $x\in\partial D\cap B_{2}$ and $f=H+R$, where $H\not\equiv 0$ is a homogeneous polynomial of degree $m$ and $\abs{R(x)}\le C\abs{x}^{m+\alpha}$. 
Then $u$ has the following properties: 
\begin{enumerate}
\renewcommand{\labelenumi}{\theenumi}
\renewcommand{\theenumi}{\rm (\alph{enumi})}
\item $u\in C_{\rm loc}^{0,1}(B_{2})$ with $\abs{u(x)} + \abs{x} \abs{\nabla u(x)} \le C\abs{x}^{m+2}$. 
\item If $v$ is any blowup limit of $u_{r}(x):=u(rx)/r^{m+2}$, then $v$ is homogeneous of degree $m+2$ and its homogeneous degree $m+2$ extension, still denoted by $v$, solves $\Delta v = H\mL^{n}\lfloor\{v\neq 0\}$ in $\mR^{n}$. 
\item \label{itm:nondegeneracy-blowup-limit} If $v$ is any blowup limit of $u_{r}(x):=u(rx)/r^{m+2}$, then $\supp\,(v)\neq\emptyset$ and for any $\epsilon\in(0,1)$ there is $c_{\epsilon}>0$ so that  
\begin{equation*}
\norm{v}_{L^{\infty}(B_{\epsilon\abs{x}}(x))} \ge c_{\epsilon} \abs{x}^{m+2} \quad \text{for all $x\in\supp\,(v)\cap\overline{B_{1/2}}$}. 
\end{equation*}
\item \label{itm:weak-flatness} If there exists a blowup limit $v$ of $u_{r}(x):=u(rx)/r^{m+2}$ such that its support is the half-space $x\cdot e\ge 0$, then for any $\delta>0$ there is $r>0$ such that $\partial D\cap B_{r} \subset \{\abs{x\cdot e}\le \delta r\}$. 
\end{enumerate}
\end{proposition}

\begin{proof}
This follows by combining \cite[Lemma~2.5]{SS25vanishingcontrast}, \Cref{lem:regularity1,lem:blowup-procedure,lem:nondegeneracy,lem:SS25-lem4.3}. 
\end{proof}

We refine the above result for our purpose: 

\begin{proposition}[Properties of blowups: II]\label{prop:properties-blowups2}
Suppose that all assumptions in \Cref{prop:properties-blowups1} hold. If we additionally assume that $H$ is harmonic and $n=2$, then the support of any blowup limit $v$ is a half space\footnote{In this case, the blowup limits have explicit form depending on $H$, see \cite[Lemma~3.5]{SS25vanishingcontrast}.}. 
\end{proposition}

\begin{proof}
Since $D$ has Lipschitz boundary, and $v$ vanishes in an exterior cone, it follows that $\supp\,(v)\neq\mR^{2}$.
On the other hand, by using \Cref{prop:properties-blowups1}\ref{itm:nondegeneracy-blowup-limit}, we see that $\supp\,(v)\neq\emptyset$. Finally, our proposition immediately follows from \Cref{lem:SS25-lem3.12}. 
\end{proof}

We can now prove the  free boundary regularity in two dimensions. 

\begin{proposition}\label{prop:2D-FB}
Suppose that all assumptions in \Cref{prop:properties-blowups2} hold. 
\begin{enumerate}
\renewcommand{\labelenumi}{\theenumi}
\renewcommand{\theenumi}{\rm (\alph{enumi})}
\item \label{itm:piecewise-C1-case} If $\partial D$ is piecewise $C^{1}$, then $\partial D$ is $C^{1}$ near $0$. 
\item \label{itm:convex-case} If $D$ is convex, then $\partial D$ is $C^{1}$ near $0$. 
\end{enumerate}
\end{proposition}

\begin{proof}
(a) Let $v$ be any blowup limit of $u_{r}(x):=u(rx)/r^{m+2}$. By using \Cref{prop:properties-blowups2}, we know that $\supp\,(v)=\{x\cdot e\ge 0\}$ for some unit vector $e$. Then the weak flatness property in \Cref{prop:properties-blowups1}\ref{itm:weak-flatness} ensures that for any $\delta>0$ there exists $r>0$ such that $\partial D\cap B_{r} \subset \{\abs{x\cdot e}\le\delta r\}$. Since $\partial D$ is piecewise $C^{1}$, then near $0$, $\partial D$ is the union of two $C^{1}$ arcs meeting at $0$ at some angle $\theta\in(0,2\pi)$. If $\theta\neq\pi$, one gets a contradiction with the weak flatness property above. This implies that $\theta=\pi$ and the tangent vectors of the two $C^{1}$ arcs are parallel at $0$, which concludes that $\partial D$ is the graph of a $C^{1}$ function near $0$. This proves case (a).

(b) As in the proof of case (a)
we already obtain the weak flatness property: for any $\delta>0$ there exists $r>0$ such that $\partial D\cap B_{r} \subset \{\abs{x\cdot e}\le\delta r\}$. Since $D$ is convex, then for each boundary point near $0$, there is a unique supporting plane. Finally, \cite[Lemma~4.4]{SS25vanishingcontrast} implies that $\partial D$ is $C^{1}$ near $0$. 
\end{proof} 

We also prove the result for edge points in dimension $n\ge 3$. 

\begin{proposition}\label{prop:edge-point} 
Suppose that all assumptions in \Cref{prop:properties-blowups1} hold. If we additionally assume that $H$ is harmonic, then $0\in\partial D$ is not an edge point. 
\end{proposition}

\begin{proof}
Suppose the contrary: $0\in\partial D$ is an edge point. After a rotation, we may arrange that the blowup of $D$ at $0$, i.e., the limit of $r^{-1}(D\cap B_{r})$ is $(S\times\mR^{n-2})\cap B_{1}$ as $r\rightarrow 0$, where $S$ is a closed sector in $\mR^{2}$ with angle $\neq\pi$. Let $v$ be any blowup limit of $u_{r}$, and we see that $C:=\supp\,(v)\subset S\times\mR^{n-2}$, and the non-degeneracy statement in \Cref{prop:properties-blowups1}\ref{itm:nondegeneracy-blowup-limit} implies that $C=S\times\mR^{n-2}$. Following exactly the same argument as in the proof of \cite[Theorem~1.10]{SS25vanishingcontrast}, which utilized Federer's reduction argument as in \cite{Weiss99FB} or \cite[Lemma~10.9]{Velichkov23FB}, we can conclude that $\partial D$ is  $C^{1}$ near $0$, which contradicts with the assumption that $0\in\partial D$ is an edge point. 
\end{proof}

We finally prove our main results.

\begin{proof}[Proof of \Cref{thm:1,thm:2}]
Without loss of generality, we may assume that $x_{0}=0$ and $\Phi\equiv\id$. The contrast $h$ is $C^{\alpha}$ near $0$ with $h(0)\neq 0$ by \eqref{eq:non-degeneracy}. Since $u^{\rm inc}$ solves $(\Delta+k^{2})u^{\rm inc}=0$ in $\mR^{n}$, then one has $u^{\rm inc}=H+R_{0}$, where $H$ is a harmonic homogeneous polynomial of order $m$ and $\abs{R_{0}(x)}\le C\abs{x}^{m+1}$. 
By using the arguments in \cite[Lemma~2.5]{SS25vanishingcontrast}, one can show that 
\begin{equation*}
-hu^{\rm inc} = h(0)H + R_{1} ,\quad \abs{R_{1}(x)} \le C\abs{x}^{m+\alpha}. 
\end{equation*}
Since $\abs{A(x)-\id}\le C\abs{x}^{2+\alpha}$ and $\abs{\nabla A(x)}\le C\abs{x}^{1+\alpha}$, then 
\begin{equation*}
\abs{\nabla\cdot (A(x)-\id)\nabla u^{\rm inc}} \le C\abs{\nabla A(x)}\abs{\nabla u^{\rm inc}} + C\abs{A(x)-\id}\abs{\nabla^{\otimes 2}u^{\rm inc}} \le C\abs{x}^{m+\alpha}, 
\end{equation*} 
where $\nabla^{\otimes 2}u^{\rm inc}$ is the Hessian matrix of $u^{\rm inc}$. Here we used the facts $\abs{\nabla u^{\rm inc}} \le C\abs{x}^{m-1}$ and $\abs{\nabla^{\otimes 2} u^{\rm inc}} \le C\abs{x}^{m-2}$. 
Since $A \in (C^{0,1}(\overline{D}))^{n\times n}$, then we conclude that
\begin{equation*}
\text{$f=h(0)H + R$ with $\abs{R(x)}\le C\abs{x}^{m+\alpha}$}.  
\end{equation*}
On the other hand, we see that 
\begin{equation*}
\abs{g(x)} \le C\abs{A(x)-\id}\abs{\nabla u^{\rm inc}} \le C \abs{x}^{m+1+\alpha}. 
\end{equation*}
After this reduction, \Cref{thm:1,thm:2} follows from the corresponding free boundary results in \Cref{prop:2D-FB,prop:edge-point}, respectively. 
\end{proof}

\section{Proof of \texorpdfstring{\Cref{thm:3,thm:4}}{Theorems \ref{thm:3} and \ref{thm:4}}}



We assume the conditions in \Cref{thm:3}. Without loss of generality, we may assume that $x_{0}=0$. 
Let $B$ denote the real orthogonal matrix specified in \ref{itm:non-decay1}. 
Since $u^{\rm inc}$ satisfies $(\Delta+k^{2})u^{\rm inc}=0$ in $\mR^{n}$, it follows that $u^{\rm inc}\circ B$ also satisfies the same equation in $\mR^{n}$. Therefore, we can write $u^{\rm inc}(Bx) - u^{\rm inc}(0) = H(x) + R_{0}(x)$, where $H\not\equiv 0$ is a harmonic homogeneous polynomial of order $m\ge 1$ and $\abs{R_{0}(x)}\le C\abs{x}^{m+1}$. 
From \ref{itm:non-decay1} it follows that 
\begin{equation*}
\left. (A(y)-\id)\nabla_{y}u^{\rm inc} \right|_{y=Bx} = c_{0}\nabla_{x} H(x) + \vec{R}(x)
\end{equation*} 
with $\abs{\vec{R}(x)}\le C\abs{x}^{m-1+\alpha}$. 
In view of \eqref{eq:scattered-field-distribution-form}, it therefore suffices, without loss of generality, to prove \Cref{thm:3} in the case $B=\id$. 

By a direct application of \Cref{lem:regularity1} (with $m+2$ being replaced by $m$) we obtain 
\begin{equation}
\abs{u^{\rm sc}(x)} + \abs{x}\abs{\nabla u^{\rm sc}(x)} \le C \abs{x}^{m}, \label{eq:Lipschitz-conti}
\end{equation} 
which naturally motivates the scaling function $\tilde{u}_{r} := u^{\rm sc}(rx)/r^{m}$, which also satisfies the estimate \eqref{eq:Lipschitz-conti} and solves the equation 
\begin{equation}
\begin{aligned} 
\nabla\cdot A(rx)\nabla \tilde{u}_{r}(x) &= (r^{-(m-2)}f(rx) - r^{2}\kappa^{2}\rho(rx)\tilde{u}_{r}(x)) \mL^{n}\lfloor \{\tilde{u}_{r}\neq 0\} \\
& \quad + r^{-(m-1)} (g\mH^{n-1}\lfloor\partial D)(rx). 
\end{aligned} \label{eq:scaling-equation} 
\end{equation}
In this setting, we say that $v$ is a blowup limit of $u^{\rm sc}$ of order $m$ at $0$ if there is a sequence $r_{j}\rightarrow 0$ so that $\tilde{u}_{r_{j}} \rightarrow v$ in $C^{0,1}(\overline{B_{1}})$ weak-$\star$.

We begin with the case $n=2$. 
Since $\partial D$ is piecewise $C^{1}$, there exists $r_{0}>0$ such that $\partial D\cap B_{r_{0}} = (\Gamma_{-} \cup \Gamma_{+})\cap B_{r_{0}}$, where the $C^{1}$ interfaces $\Gamma_{\pm}$ intersect at $x_{0}=0$. Without loss of generality, we may further assume that $\Gamma_{-}\cap B_{r_{0}} \subset \{x_{1}\le 0\}$ and $\Gamma_{+}\cap B_{r_{0}} \subset \{x_{1}\ge 0\}$. 
Since the normal vector $\nu$ is a continuous vector field on each $\Gamma_{\pm}$, we can define
\begin{equation*}
\nu_{-}:=\lim_{\Gamma_{1}\ni x\rightarrow 0}\nu(x) ,\quad \nu_{+}:=\lim_{\Gamma_{2}\ni x\rightarrow 0}\nu(x), 
\end{equation*}
and introduce the straight lines $\tilde{\Gamma}_{-}\subset\{x_{1}\le 0\}$ and $\tilde{\Gamma}_{+}\subset\{x_{1}\ge 0\}$, each perpendicular to $\nu_{-}$ and $\nu_{+}$, respectively. Any blowup limit $v \in C^{0,1}(\ol{B}_1)$ of order $m$ of $u^{\rm sc}$ solves the equation 
\begin{equation}
\Delta v = c_{0} \sum_{\pm} (\nu_{\pm}\cdot\nabla H) \mH^{1}\lfloor\tilde{\Gamma}_{\pm} \quad \text{in $B_{1}$.}  \label{eq:blowup-v-2}
\end{equation}
At this stage, it is not yet known whether $v$ is homogeneous of degree $m$. However, this fact is not needed for the proof of the results in the case $n=2$.

First, all half-space solutions can be characterized directly by the same arguments as in \cite[Lemmas~3.3 and 3.5]{SS25vanishingcontrast}: 

\begin{lemma}\label{lem:half-space-solution}
There exists a unique blowup limit $v$ that satisfies \eqref{eq:blowup-v-2} with $\tilde{\Gamma}_{\pm}=\{x_{2}=0\}\cap\{\pm x_{1}\ge 0\}$ and $\supp\,(v)=\{x_{2}\ge 0\}$. For any $m\ge 1$, we write $H(re^{\bfi\theta})=ar^{m}e^{\bfi m\theta}+br^{m}e^{-\bfi m\theta}$ for some constants $a,b\in\mC$, then the unique solution $v$ is a polynomial of order $m$ given explicitly by 
\begin{equation}
v(re^{\bfi\theta}) = c_{0}(a-b)r^{m} \left( e^{\bfi m\theta} - e^{-\bfi m\theta} \right). \label{eq:explicit-formula}
\end{equation}
for all $(x_{1},x_{2})\cong x_{1}+\bfi x_{2} = re^{\bfi\theta}$ with $0 < \theta < \pi$ and $0<r<1$. 
\end{lemma} 

\begin{remark*}
It is natural to extend \eqref{eq:explicit-formula} to all $r>0$, and this extension satisfies 
\begin{equation}
\Delta v = c_{0} \sum_{\pm} (\nu_{\pm}\cdot\nabla H) \mH^{1}\lfloor\tilde{\Gamma}_{\pm} \quad \text{in $\mR^{2}$.} \label{eq:blowup-v-2-extended}
\end{equation}
\end{remark*}




\begin{remark}\label{rem:explaination-why-orthogonal} 
We now explain the difficulty that arises if we replace $B$ in \Cref{thm:3}\ref{itm:non-decay1} with a general invertible matrix. 
Since $u^{\rm inc}$ satisfies $(\Delta+k^{2})u^{\rm inc}=0$ in $\mR^{n}$, it follows that $u_{B}^{\rm inc}:=u^{\rm inc}\circ B$ satisfies 
\begin{equation*}
\nabla\cdot \left[ B^{-1}(B^{-1})^{\intercal} \nabla u_{B}^{\rm inc} \right] + k^{2} u_{B}^{\rm inc} = 0 \quad \text{in $\mR^{n}$.} 
\end{equation*}
Therefore, we can write $u^{\rm inc}(Bx) - u^{\rm inc}(0) = H(x) + R_{0}(x)$, where $H\not\equiv 0$ is a harmonic homogeneous polynomial of order $m\ge 1$ and $\abs{R_{0}(x)}\le C\abs{x}^{m+1}$. 
From \Cref{thm:3}\ref{itm:non-decay1} it follows that 
\begin{equation*}
\left. (A(y)-\id)\nabla_{y}u^{\rm inc} \right|_{y=Bx} = c_{0}\nabla_{x} H(x) + \vec{R}(x)
\end{equation*} 
with $\abs{\vec{R}(x)}\le C\abs{x}^{m-1+\alpha}$. 
Using \eqref{eq:scattered-field-distribution-form}, we see that the function $u_{B}^{\rm sc} := u^{\rm sc}\circ B$ satisfies 
\begin{equation*}
\nabla\cdot\left[ B^{-1}A(Bx)(B^{-1})^{\intercal}\nabla u_{B}^{\rm sc} \right] + \rho(Bx)u_{B}^{\rm sc} = f_{B}\mL^{n}\lfloor B^{-1}(D) + g_{B}\mH^{n-1}\lfloor \partial(B^{-1}(D)), 
\end{equation*}
where 
\begin{equation*}
f_{B} = - \nabla\cdot\left[ B^{-1}(A(Bx)-\id)(B^{-1})^{\intercal}\nabla u_{B}^{\rm inc} \right] - h(Bx)u_{B}^{\rm inc} 
\end{equation*}
and 
\begin{equation*}
g_{B}=\nu(Bx)\cdot\left. (A(y)-\id)\nabla_{y}u^{\rm inc} \right|_{y=Bx}. 
\end{equation*}
However, now $B^{-1}A(Bx)(B^{-1})^{\intercal}\rightarrow B^{-1}(B^{-1})^{\intercal}$ as $x\rightarrow 0$, and \eqref{eq:blowup-v-2} becomes 
\begin{equation*}
\nabla\cdot(B^{-1}(B^{-1})^{\intercal}\nabla v) = c_{0} \sum_{\pm} (\nu_{\pm}\cdot\nabla H) \mH^{1}\lfloor\tilde{\Gamma}_{\pm} \quad \text{in $B_{1}$,} 
\end{equation*}
which makes the subsequent computations in \Cref{lem:half-space-solution} considerably more complicated. 
\end{remark}


\begin{proof}[Proof of \Cref{lem:half-space-solution}] 
We verify that $v$ is a harmonic polynomial of degree $m$ of the form $v=cr^{m}e^{\bfi m\theta}+dr^{m}e^{-\bfi m\theta}$. Imposing the Dirichlet boundary condition at $\theta=0$ (or $\theta=\pi$) gives $d=-c$, so 
\begin{equation*}
v(re^{\bfi\theta}) = cr^{m} \left( e^{\bfi m\theta} - e^{-\bfi m\theta} \right). 
\end{equation*}
Finally, imposing the Neumann boundary condition at $\theta=0$ (or $\theta=\pi$) gives 
\begin{equation*}
\begin{aligned} 
& \bfi m cr^{m} = \left. \bfi m cr^{m}\left( e^{\bfi m\theta} + e^{-\bfi m\theta} \right) \right|_{\theta=0} = \left. \partial_{\theta} v(re^{\bfi\theta}) \right|_{\theta=0} \\
& \quad = c_{0} \left. \partial_{\theta} H(re^{\bfi\theta}) \right|_{\theta=0} = c_{0} \left. \bfi m \left( ar^{m}e^{\bfi m\theta} - br^{m}e^{-\bfi m\theta} \right) \right|_{\theta=0} = c_{0} \bfi m r^{m}(a-b), 
\end{aligned} 
\end{equation*}
which conclude our lemma. 
\end{proof}

It is important to note that the blowup limit $v$ satisfies $\Delta v =  g\lfloor\tilde{\Gamma}_{\pm}$ with the special property that $g$ can be expressed in terms harmonic homogeneous polynomial. This structure allows us to rule out blowup solutions supported in sectors of angle $\neq \pi$ and is crucial for their characterization, using the same ideas as in \cite[Lemma~3.9]{SS25vanishingcontrast}: 

\begin{lemma}\label{lem:blowup-rational-angle}
Let $\theta_{0}\in(0,2\pi)\setminus \pi\mQ$, let $C=\{re^{\bfi\theta} : r>0 , \theta\in(0,\theta_{0}) \}$, and suppose that $w$ solves 
\begin{equation*}
\Delta w = 0 \text{ in $C$} ,\quad w|_{\theta=0}=0 ,\quad w|_{\theta=\theta_{0}}=0
\end{equation*}
as well as the Bernoulli condition 
\begin{equation*}
\left. \partial_{\theta}w \right|_{\theta=0} = c_{0} \left. \partial_{\theta}H(re^{\bfi\theta}) \right|_{\theta=0} ,\quad \left. \partial_{\theta}w \right|_{\theta=\theta_{0}} = c_{0} \left. \partial_{\theta}H(re^{\bfi\theta}) \right|_{\theta=\theta_{0}}, 
\end{equation*}
where $H(re^{\bfi\theta})=ar^{m}e^{\bfi m\theta}+br^{m}e^{-\bfi m\theta}$ is a harmonic homogeneous polynomial of degree $m\ge 1$ on $\mR^{2}$. Then $w\equiv 0$ and $H \equiv 0$. 
\end{lemma}

\begin{proof}
Note that the solution in \Cref{lem:half-space-solution} also extends to $0<\theta<2\pi$, corresponding to an antipodally even or odd extension when $m$ is even or odd, respectively. By unique continuation property, the antipodal extension of the solution $v$ in \Cref{lem:half-space-solution} is the unique solution to 
\begin{equation}
\Delta w = 0 \text{ in $C$} ,\quad w|_{\theta=0}=0 ,\quad \left. \partial_{\theta}w \right|_{\theta=0} = c_{0} \left. \partial_{\theta}H(re^{\bfi\theta}) \right|_{\theta=0}. \label{eq:antipodal-extension}
\end{equation} 
From the condition $w|_{\theta=\theta_{0}}=0$, we see that 
\begin{subequations} \label{eq:linear-system1}
\begin{equation}
(a-b)\left( e^{\bfi m\theta_{0}} - e^{-\bfi m\theta_{0}} \right) = 0 \quad \text{(choosing $r=1$)},  
\end{equation}
and from the condition $\left. \partial_{\theta}w \right|_{\theta=\theta_{0}} = c_{0} \left. \partial_{\theta}H(re^{\bfi\theta}) \right|_{\theta=\theta_{0}}$, we see that 
\begin{equation*}
(a-b) \left( e^{\bfi m\theta_{0}} + e^{-\bfi m\theta_{0}} \right) = ae^{\bfi m\theta_{0}} - b e^{-\bfi m\theta_{0}} \quad \text{(choosing $r=1$)},  
\end{equation*}
equivalently, 
\begin{equation}
-b e^{\bfi m\theta_{0}} + ae^{-\bfi m\theta_{0}} = 0. 
\end{equation}
\end{subequations} 
The equations in \eqref{eq:linear-system1} can be written in matrix form as 
\begin{equation}
\left(\begin{array}{cc} e^{\bfi m\theta_{0}} - e^{-\bfi m\theta_{0}} & -e^{\bfi m\theta_{0}} + e^{-\bfi m\theta_{0}} \\ e^{-\bfi m\theta_{0}} & -e^{\bfi m\theta_{0}} \end{array}\right) \left(\begin{array}{c} a \\ b  \end{array}\right) = 0. 
\end{equation}
We compute that 
\begin{equation*}
\begin{aligned}
& \det \left(\begin{array}{cc} e^{\bfi m\theta_{0}} - e^{-\bfi m\theta_{0}} & -e^{\bfi m\theta_{0}} + e^{-\bfi m\theta_{0}} \\ e^{-\bfi m\theta_{0}} & -e^{\bfi m\theta_{0}} \end{array}\right) \\ 
& \quad = -e^{2\bfi m\theta_{0}} + 2 - e^{-2\bfi m\theta_{0}} = -(e^{\bfi m\theta_{0}} - e^{-\bfi m\theta_{0}})^{2} \\ 
& \quad = -(2\bfi\sin(m\theta_{0}))^{2} = 4 \sin^{2}(m\theta_{0}) \neq 0 
\end{aligned}
\end{equation*}
since $\theta_{0}\in(0,2\pi)\setminus \pi\mQ$, therefore we conclude that $a=b=0$, which is our lemma. 
\end{proof}

We now ready to prove our main result.

\begin{proof}[Proof of \Cref{thm:3}]
Suppose, to the contrary, that the medium does not scatter for some incident wave $u^{\rm inc}\not\equiv 0$. Then, by \Cref{lem:blowup-rational-angle}, the blowup limit $v$ of $u^{\rm sc}$ of order $m$ at $0$ must vanish identically. From \eqref{eq:blowup-v-2-extended}, it follows that  
\begin{equation*}
(\nu_{\pm}\cdot\nabla H) \mH^{n-1}\lfloor\tilde{\Gamma}_{\pm} = 0. 
\end{equation*}
Without loss of generality, we may assume $x_{0}=0$, $\tilde{\Gamma}_{+}=\{re^{\bfi\theta}:r>0 , \theta=0\}$ and $\tilde{\Gamma}_{-}=\{re^{\bfi\theta}:r>0 , \theta=\theta_{0}\}$. In this case, 
\begin{equation*}
\left. \partial_{\theta} H(re^{\bfi\theta}) \right|_{\theta=0} = \left. \partial_{\theta} H(re^{\bfi\theta}) \right|_{\theta=\theta_{0}} =0. 
\end{equation*}
Any harmonic homogeneous polynomial of order $m\ge 1$ can be written as $H(re^{\bfi\theta})=ar^{m}e^{\bfi m\theta}+br^{m}e^{-\bfi m\theta}$ for some constants $a,b\in\mC$. The condition $\left. \partial_{\theta} H(re^{\bfi\theta}) \right|_{\theta=0}=0$ gives $b=a$, so 
\begin{equation*}
H(re^{\bfi\theta})=a(r^{m}e^{\bfi m\theta}+r^{m}e^{-\bfi m\theta}). 
\end{equation*}
Then the condition $\left. \partial_{\theta} H(re^{\bfi\theta}) \right|_{\theta=\theta_{0}} =0$ yields 
\begin{equation*}
0 = a\left(e^{\bfi m\theta_{0}}-e^{-\bfi m\theta_{0}}\right) = 2\bfi a\sin(m\theta_{0}). 
\end{equation*}
Since $\theta_{0}\notin \pi\mQ$, we must have $a=0$, hence $H\equiv 0$, contradicting the fact that $H$ is a harmonic homogeneous polynomial of order $m\ge 1$. 
\end{proof}

Next, we consider the case $n\ge 3$. In an earlier work by two of the current authors, the homogeneity of the blowup limit was needed to apply the well-known  Federer's dimension reduction argument. This required the introduction of a balanced energy functional and its monotonicity. In the Bernoulli case this is slightly more complicated, and would need a technical computation; one such possible monotonicity formula can be found in  \cite{CSY18FB}.
Since the situation in our case (and also in the work of \cite{SS25vanishingcontrast}) is somewhat easier, one can avoid such a monotonicity functional, and circumvent the issue using the simple geometric situation.\footnote{It should be remarked that the use of monotonicity formula may actually enhance the main result and achieve a stronger version of the theorem. We may come back to this in the future, as it is too technical for the current paper. }
The idea is to use   the unique continuation property, adapting the ideas in \cite[Lemma~3.3]{SS25vanishingcontrast}:

\begin{lemma}\label{lem:k-homogeneous}
Let $n\ge 3$ and let 
\begin{equation*}
\tilde{\Gamma}_{\pm}=\{re^{\bfi\theta_{\pm}}\in \mC\cong\mR^{2} : r\ge 0\}\times\mR^{n-2} 
\end{equation*}
with distinct $\theta_{\pm}\in [0,2\pi)$ and $\nu_{\pm}$ being the corresponding unit normal vectors. Let $w \in C^{0,1}(\ol{B}_1)$ be a solution to 
\begin{equation}
\Delta w = c_{0} \sum_{\pm} (\nu_{\pm}\cdot \nabla P) \mH^{n-1}\lfloor\tilde{\Gamma}_{\pm} \quad \text{in $B_{1}\subset\mR^{n}$,} \label{eq:blowup-federer}
\end{equation}
and assume that $w=0$ in one connected component of $B_{1}\setminus(\tilde{\Gamma}_{+}\cup\tilde{\Gamma}_{-})$, where $P$ is a homogeneous polynomial of order $k$. Then $w$ is a homogeneous polynomial of order $k$. 
\end{lemma}

\begin{proof}
Since both the Laplacian and polynomial homogeneity are invariant under orthogonal transformations, it suffices to prove the lemma in the case $\theta_{+}=\pi/2$, that is, 
\begin{equation*}
\tilde{\Gamma}_{+}=\{0\}\times\mR_{\ge 0}\times\mR^{n-2}. 
\end{equation*}
Let $D$ and $V$ be the connected components of $B_{1}\setminus(\tilde{\Gamma}_{+}\cup\tilde{\Gamma}_{-})$ such that $w = 0$ in $V$. Without loss of generality, assume that $w$ vanishes to the left of $\tilde{\Gamma}_+$ (i.e.\ $V$ contains $B(\frac{1}{2}e_2, \delta) \cap \{ x_1 < 0 \}$ for some $\delta > 0$). As in the beginning of \Cref{sec_2}, it follows that 
\begin{equation}
\Delta w = 0 \text{ in $D \cap \{ x_1 > 0 \}$} ,\quad w|_{B_1 \cap \tilde{\Gamma}_+} = 0, \quad \partial_{1}w|_{B_{1}\cap\tilde{\Gamma}_{+}}=-c_0 \partial_{1}P|_{B_{1}\cap\tilde{\Gamma}_{+}}. \label{eq:Cauchy-problem}
\end{equation}
By the unique continuation property, there is at most one solution to \eqref{eq:Cauchy-problem}. 
By the Cauchy-Kowalevski theorem, there exists an analytic solution 
\begin{equation*}
w(x) = \sum_{\ell=0}^{\infty}\frac{\partial_{1}^{\ell}w|_{\tilde{\Gamma}_{+}}}{\ell!}x_{1}^{\ell} = x_{1}\partial_{1}P|_{\tilde{\Gamma}_{+}} + \sum_{\ell=2}^{\infty}\frac{\partial_{1}^{\ell}w|_{\tilde{\Gamma}_{+}}}{\ell!}x_{1}^{\ell}
\end{equation*}
for all $x\in U$ with $x_{1}>0$. 
We compute that 
\begin{equation*}
\begin{aligned}
0 &= \Delta_{x'}w + \partial_{1}^{2}w = x_{1} \Delta_{x'}\partial_{1}P|_{\tilde{\Gamma}_{+}}  + \sum_{\ell=2}^{\infty}\frac{\Delta_{x'}\partial_{1}^{\ell}w|_{\tilde{\Gamma}_{+}}}{\ell!}x_{1}^{\ell} + \sum_{\ell=0}^{\infty} \frac{\partial_{1}^{\ell+2}w|_{\tilde{\Gamma}_{+}}}{\ell!}x_{1}^{\ell} \\ 
&= \partial_{1}^{2}w|_{\tilde{\Gamma}_{+}} + \left(\Delta_{x'}\partial_{1}P|_{\tilde{\Gamma}_{+}} + \partial_{1}^{3}w|_{\tilde{\Gamma}_{+}} \right) x_{1} + \sum_{\ell=2}^{\infty} \frac{(\Delta_{x'}\partial_{1}^{\ell}w + \partial_{1}^{\ell+2}w)|_{\tilde{\Gamma}_{+}}}{\ell!}x_{1}^{\ell}
\end{aligned}
\end{equation*}
for all $x\in U$ with $x_{1}>0$. Comparing coefficients, we obtain 
\begin{equation*}
\partial_{1}^{2\ell}w|_{\tilde{\Gamma}_{+}}=0 ,\quad \partial_{1}^{2\ell+1}w|_{\tilde{\Gamma}_{+}}=(-1)^{\ell}\Delta_{x'}^{\ell}\partial_{1}P|_{\tilde{\Gamma}_{+}} \quad \text{for all $\ell\in\mN$}. 
\end{equation*}
Since $\partial_{1}P|_{\tilde{\Gamma}_{+}}$ is a polynomial of order $k-1$, then 
\begin{equation*}
\partial_{1}^{2\ell+1}w|_{\tilde{\Gamma}_{+}}=0 \quad \text{for all $\ell\in\mN$ with $2\ell+1>k$.}
\end{equation*}
Putting the above discussions together, we conclude 
\begin{equation}
w(x) = -c_0 \sum_{\ell=0}^{\lfloor\frac{k-1}{2}\rfloor}\frac{(-1)^{\ell}\Delta_{x'}^{\ell}\partial_{1}P|_{B_{1}\cap\tilde{\Gamma}_{+}}}{(2\ell+1)!}x_{1}^{2\ell+1} \quad \text{in $D \cap \{ x_{1}>0 \}$} \label{eq:homogeneous-explicit}
\end{equation}
is the unique solution to \eqref{eq:Cauchy-problem}. Again, by the unique continuation property, the solution to \eqref{eq:blowup-federer} must coincide, up to an orthogonal transformation, with \eqref{eq:homogeneous-explicit} in $D$, which establishes the result. 
\end{proof}

We are now in a position to prove \Cref{thm:4}. 

\begin{proof}[Proof of \Cref{thm:4}] 
Suppose, to the contrary, that the medium does not scatter for some incident wave $u^{\rm inc}\not\equiv 0$. The first two steps are the same as in the proof of the two-dimensional case (\Cref{thm:3}): 
\begin{itemize}
\item First, one obtains the Lipschitz regularity of $u^{\rm sc}$ as in \eqref{eq:Lipschitz-conti}; 
\item Next, one sees that any blowup limit $v$ satisfies 
\begin{equation*}
\Delta v = c_{0} \sum_{\pm} (\nu_{\pm}\cdot \nabla H) \mH^{n-1}\lfloor\tilde{\Gamma}_{\pm} \quad \text{in $B_{1}\subset\mR^{n}$,} 
\end{equation*}
and $v=0$ in a component of $B_{1}\setminus(\tilde{\Gamma}_{+}\cup\tilde{\Gamma}_{-})$. 
\end{itemize}
Without loss of generality, assume $\theta_{+}=0$. We now apply Federer's dimension reduction argument, as mentioned earlier. Fix a point $e=e_{n}\in\Gamma_{+}\cap\Gamma_{-}$. Then $H$ has a zero at $e$ of order $1 \le k \le m$. By a direct application of \Cref{lem:regularity1}, 
\begin{equation*}
\abs{v(z+e)} + \abs{z}\abs{\nabla v(z+e)} \le C \abs{z}^{k},
\end{equation*}
which naturally motivates the scaling function $v_{r} := v(rz+e)/r^{k}$. Next, the blowup limit $w$ of $v_{r}$ satisfies \eqref{eq:blowup-federer} for some harmomic homogeneous polynomial of order $k$ satisfying $H(e+z)=P(z)+O(\abs{z}^{k+1})$, 
and $w=0$ in a component of $B_{1}\setminus(\tilde{\Gamma}_{+}\cup\tilde{\Gamma}_{-})$. Using \Cref{lem:k-homogeneous}, $w$ is homogeneous or order $k$. Following the arguments in the proof of \cite[Theorem~1.10]{SS25vanishingcontrast}, one can show that $w$ is independent of $z_{n}$, reducing the problem to $\mR^{n-1}$. Iterating this procedure eventually reduces the problem to $\mR^{2}$, after which the remaining steps follow from the proof of \Cref{thm:3}.
\end{proof}

\section*{Acknowledgments}

Kow was supported by the National Science and Technology Council of Taiwan, NSTC 112-2115-M-004-004-MY3.
Salo was partly supported by the Research Council of Finland (Centre of Excellence in Inverse Modelling and Imaging and FAME Flagship, grants 353091 and 359208). 
Shahgholian was supported by Swedish Research Council (grant no. 2021-03700).

\section*{Declarations}

\noindent {\bf  Data availability statement:} All data needed are contained in the manuscript.

\medskip
\noindent {\bf  Funding and/or Conflicts of interests/Competing interests:} The authors declare that there are no financial, competing or conflict of interests.

\end{sloppypar}

\bibliographystyle{custom}
\bibliography{ref}
\end{document}